\documentclass[11pt,a4paper]{article}

\usepackage[english,german]{babel}
\usepackage{latexsym}
\usepackage{amsmath}
\usepackage{amssymb}
\usepackage{graphics}
\usepackage{graphicx}
\usepackage{fancyhdr}
\usepackage{color}
\usepackage{epstopdf}
\usepackage{amsfonts}
\usepackage[utf8]{inputenc}
\usepackage[numbers]{natbib}				
\renewcommand{\cite}[1]{\citeauthor*{#1} [\citeyear{#1}]}
\usepackage{amsthm,url,xspace}
\definecolor{couleurCitations}{rgb}{0,0.65,0}
\definecolor{couleurRef}{rgb}{0.75,0,0}
\usepackage[pdftex,colorlinks=true,
linkcolor=couleurRef,citecolor=blue,bookmarks=true,plainpages=true,
 urlcolor= blue] {hyperref}

\usepackage{float}
\usepackage[caption=false]{subfig}
\usepackage{bm}

\usepackage[linesnumbered,ruled]{algorithm2e}
\usepackage{colortbl}

\title{\textbf{Title}}                                   %
\makeatletter                                            %
\renewcommand\theequation{\thesection.\arabic{equation}} %
\@addtoreset{equation}{section}                          %
\makeatother                                             %

\begin{document}
\newcommand\nrthis{\addtocounter{equation}{1}\tag{\theequation}}
\newcommand*\tageq{\refstepcounter{equation}\tag{\theequation}}

    %
    \newcommand{\estimatorb}[1]{\hat{#1}}			



\theoremstyle{plain}
\newtheorem{theorem}{Theorem}
\newtheorem{proposition}{Proposition}
\newtheorem{lemma}{Lemma}
\newtheorem{definition}{Definition}
\newtheorem*{remark}{Remark}
\newtheorem{assumptions}{Assumption}
\newtheorem{properties}{Properties}
\newtheorem{corollary}{Corollary}
\newtheorem{example}{Example}

\newcommand{\op}[1]{\operatorname{#1}}                     
\newcommand{\1}{{\rm 1}\mskip -4,5mu{\rm l} }              
\newcommand{\amin}{\mathop{\mathrm{arg\,min}}}             
\newcommand{\pen}{\op{pen}}                                
\newcommand{\h}[1]{\hat{#1}}                               
\newcommand{\bh}{\hat{\beta}}                               
\newcommand{\bt}{\widetilde{\beta}}                               
\newcommand{\bl}{\bar{\beta}}                               
\newcommand{\bc}{\stackrel{\circ}{\beta}}                               

\newcommand{\eh}{\hat{\epsilon}}                             
\newcommand{\bha}{\hat{\beta}_{\alpha}}                       
\newcommand{\bo}{\beta^{0}}                                    

\newcommand{\0}[1]{{#1}^{0}}                               
\newcommand{\dO}{\delta_{\Omega}}                          
\newcommand{\GO}{\Gamma_{\Omega}}                          
\newcommand{\Obs}{\Omega(\beta_{S})}    
\newcommand{\Rp}{\mathbb{R}^{p}}                           
\newcommand{\R}[1]{\mathbb{R}^{#1}}                        
\newcommand{\Ost}{\Omega_{*}}                              
\newcommand{\N}{\mathcal{N}}                               

\newcommand{\dell}{\delta\ltsr{\epsilon}\left[(\lambda^{*}+\lambda^{m})\Omega(\bh_{S}-\beta)+(\lambda^{*}-\lambda^{m})\Omega^{S^{c}}(\bh_{S^{c}})\right]}                              
\newcommand{\ltn}[1]{  \lVert  #1    \rVert  _{n}^{2}}            
\newcommand{\lt}[1]{  \lVert  #1    \rVert_{\ell_{2}}}  
\newcommand{\ltsr}[1]{ \lVert  #1    \rVert  _{n}}        
\newcommand{\ltsrr}[1]{\lVert  #1    \rVert  _{n}}              
\newcommand{\slt}[2]{  \langle #1,#2 \rangle_n}          
\newcommand{\lo}[1]{  \lVert  #1    \rVert  _{1}}            
\newcommand{\linf}[1]{  \lVert  #1    \rVert  _{\infty}}            

\newcommand{\Var}{\operatorname{Var}}
\newcommand{\E}{\operatorname{E}}
\newcommand{\Prob}{\operatorname{P}}

\newcommand{\supp}{\mathop{\mathrm{supp}}}

\newcommand{\cc}{\cellcolor[rgb]{0.9,1,0.8}}

\selectlanguage{english}  %

\title{\bf Sharp Oracle inequalities for square root regularization.}
\author{\bf \small Benjamin Stucky and Sara van de Geer}

\maketitle
\begin{center}
{\textit{Seminar für Statistik}} \\
{\textit{ETH Zürich}}\\
{\textit{Switzerland}}
\end{center}

\begin{abstract} 
 We study a set of regularization methods for high-dimensional linear regression models. These penalized estimators have the square root of the residual sum of squared errors as loss function, 
 and any weakly decomposable norm as penalty function. This fit measure is chosen because of its property that the estimator does not depend on the unknown 
 standard deviation of the noise. On the other hand, a generalized weakly decomposable norm penalty is very useful in being able to deal with different underlying sparsity structures. 
 We can choose a different sparsity inducing norm
 depending on how we want to interpret the unknown parameter vector $\beta$. Structured sparsity norms, as defined in \citet{mi2}, are special cases of weakly decomposable norms, 
 therefore we also include the square root LASSO (\citet{belloni1}), the group square root LASSO (\citet{led1}) 
 and a new method called the square root SLOPE (in a similar fashion to the SLOPE from \citet{candes1}).
 For this collection of estimators our results provide sharp oracle inequalities with the Karush-Kuhn-Tucker conditions. We discuss some examples of estimators. Based on a simulation we 
 illustrate some advantages of the square root SLOPE.
\end{abstract}

\smallskip
\noindent \textit{Square Root LASSO, Structured Sparsity, Karush-Kuhn-Tucker, Sharp Oracale Inequality, Weak Decomposability.}

\section{Introduction and Model}
The recent development of new technologies makes data gathering not a big problem any more. In some sense there is more data than we can handle, or than we need. 
The problem has shifted towards finding useful and meaningful information in the big sea of data. An example where such problems arise is the high-dimensional linear regression model
\begin{equation}Y=X\beta^{0}+\epsilon. \label{lr}\end{equation}
Here $Y$ is the $n-$dimensional response variable, $X$ is the $n\times p$ design matrix and $\epsilon$ is the identical and independent distributed noise vector. 
The noise has $\mathrm{E}(\epsilon_{i})=0, \mathrm{Var}(\epsilon_{i})=\sigma^{2},\ \ \forall i\in \{1,...,n\}$. 
Assume that $\sigma$ is \textbf{unknown}, and that $\beta^{0}$ is the "true" underlying  $p-$dimensional parameter vector of the linear regression model with active set $S_{0}:=\supp (\beta^{0})$.\\

While trying to explain $Y$ through different other variables, in the \\
high-dimensional linear regression model, we need to set less important explanatory variables to zero. Otherwise we would have 
overfitting. 
This is the process of finding a trade-off between a good fit and a sparse solution. In other words we are trying to find a solution that explains our data well, 
but at the same time only uses more important variables to do so. 

The most famous and widely used estimator for the high-dimensional regression
model is the $\ell_1-$regularized version of least squares, called LASSO (\citet{tib1})
$$\hat{\beta}_{L}(\sigma):=\amin_{\beta\in\mathbb{R}^{p}}\left\{ \ltn{Y-X\beta}+2\lambda_{1}\sigma\lo{\beta} \right\}.$$

Here $\lambda_{1}$ is a constant called the regularization level, which regulates how sparse our solution should be. Also note that the construction of the LASSO estimator depends on the unknown noise level $\sigma$.
We moreover let $\lo{a}:=\sum_{i=1}^{p}\left|a_{i}\right|$ for any $a\in{\mathbb{R}^{p}}$ denote the $\ell_1-$norm and for any $a\in\mathbb{R}^{n}$ we write
$\ltn{a}=\sum_{j =1}^{n}a_{j}^{2}/n$, the $\ell_{2}-$ norm squared and divided by $n$.
The LASSO uses the $\ell_1-$norm as a measure of sparsity. This measure as regulizer sets a number of parameters to zero.

Let us rewrite the LASSO into the following form
$$\hat{\beta}_{L}=\amin_{\beta\in\mathbb{R}^{p}}\left\{ \left(\ltsr{Y-X\beta}+\lambda^{'}(\beta)\lo{\beta}\right)\cdot \frac{2\lambda_{1}\sigma}{\lambda^{'}(\beta)} \right\},$$
where $\lambda^{'}(\beta):=\frac{2\lambda_{1}\sigma}{\ltsr{Y-X\beta}}.$
Instead of minimizing with $\lambda^{'}(\beta)$, a function of $\beta$, let us assume that we keep $\lambda^{'}( \beta)$ a fixed constant.
Then we get the Square Root LASSO method
$$\hat{\beta}_{srL}:=\amin_{\beta\in\mathbb{R}^{p}}\left\{ \ltsr{Y-X\beta}+\lambda\lo{\beta} \right\}.$$
So in some sense the $\lambda$ for the Square Root LASSO is a scaled version, scaled by an adaptive estimator of $\sigma$, of $\lambda_{1}$ from the LASSO.
By the optimality conditions it is true that
$$\bh_{L}(\ltsr{Y-X\bh_{srL}})=\bh_{srL}.$$
The Square Root LASSO was introduced by \citet{belloni1} in order to get a pivotal method. An equivalent formulation as a joint convex optimization program can be found in \citet{owen1}. 
This method has been studied under the name Scaled LASSO in \citet{sun1}.
Pivotal means that the theoretical $\lambda$ does not depend on the unknown standard deviation $\sigma$
or on any estimated version of it. The estimator does not require the estimation of the unknown $\sigma$. 
\citet{belloni2} also showed that under Gaussian noise the theoretical $\lambda$ can be chosen of order $\Phi^{-1}(1-\alpha/2p)/\sqrt{n-1}$, with $\Phi^{-1}$ denoting the inverse of the standard Gaussian 
cumulative distribution function, and $\alpha$ being some small probability. 
This is independent of $\sigma$ and achieves a near oracle inequality for the prediction norm of convergence rate $\sigma\sqrt{(|S_{0}|/n)\log{p}}$. 
In contrast to that, the theoretical penalty level of the LASSO depends on knowing $\sigma$ in order to achieve similar oracle inequalities for the prediction norm.

The idea of the square root LASSO was further developed in \citet{led1} to the group square root LASSO, in order to get a selection of groups of predictors. The group LASSO norm is another way to describe
an underlying sparsity, namely if groups of parameters should be set to zero, instead of individual parameters. Another extension for the the square root LASSO in the case
of matrix completion was given by \citet{klopp}.\\
Now in this paper we go further and generalize the idea of the square root LASSO to any sparsity inducing norm.
From now on we will look at the family of norm penalty regularization methods, which are of the following square root type
$$\hat{\beta}:=\amin_{\beta\in\mathbb{R}^{p}}\left\{ \ltsr{Y-X\beta}+\lambda\Omega(\beta) \right\},$$
where $\Omega$ is any norm on $\mathbb{R}^{p}$.
This set of regularization methods will be called square root regularization methods.
Furthermore, we introduce the following notations
\begin{align*}
 &\hat{\epsilon}:=Y-X\bh &&\text{the residuals,}&&\qquad\\
 &\Omega^{*}(x):=\max_{z,\Omega(z)\leq 1}z^{T}x, \text{ }x\in \mathbb{R}^{p}  &&\text{the dual norm of the norm } \Omega, \text{ and }&&\qquad\quad\quad \\
 &\beta_{S}=\{\beta_{j}:j\in S\} && \forall S \subset \{1,...,p\} \text{ and all vectors }\beta\in \mathbb{R}^{p}.&&\qquad\\
\end{align*}

Later we will see that describing the underlying sparsity with an appropriate sparsity norm can make a difference in how good the errors will be. Therefore in this paper we extend the idea of 
the square root LASSO with the $\ell_1-$penalty to more general weakly decomposable norm penalties. The theoretical $\lambda$ of such an estimator will not depend on $\sigma$ either. We introduce the Karush-Kuhn-Tucker conditions for these estimators and give
sharp oracle inequalities. In the last two sections we will give some examples of different norms and simulations comparing the square root LASSO with the square root SLOPE.\\

\section{Karush-Kuhn-Tucker Conditions} 
As we already have seen before, these estimators need to calculate a minimum over $\beta$. The Karush-Kuhn-Tucker conditions characterize this minimum.
In order to formulate these optimality conditions we need some concepts of convex optimization. For the reader who is not familiar with this topic, we will introduce the subdifferential, 
which generalizes the differential, and give a short overview of some properties, as can be found for example in \citet{bach2}.
For any convex function $g:\R{p}\to \R{}$ and any vector $w\in \R{p}$ we define its subdifferential as
$$\partial g (w):= \{ z \in \R{p} ; \text{ }g(w')\geq g(w)+z^{T}(w'-w) \text{  }\forall w'\in \R{p}\} .$$
The elements of $\partial g (w)$ are called the subgradients of $g$ at $w$.

Let us remark that all convex functions have non empty subdifferentials at every point. Moreover by the definition of the subdifferential
any subgradient defines a tangent space $w'\mapsto g(w)+z^{T}\cdot (w'-w)$, that goes through $g(w)$ and is at any point lower than the function $g$. 
If $g$ is differentiable at $w$, then its subdifferential at $w$ is the usual gradient.
Now the next lemma, which dates back to Pierre Fermat (see \citet{fermat}), shows how to find a global minimum for a convex function $g$.

\begin{lemma}[Fermat's Rule]\label{l0}
 For all convex functions $g:\R{p}\to \R{}$ it holds that
 $$v\in \R{p} \text{ is a global minimum of } g \text{ }\Leftrightarrow 0\in \partial g(v).$$
\end{lemma}

For any norm $\Omega$ on $\mathbb{R}^{p}$ with $\omega\in\mathbb{R}^p$ it holds true that its subdifferential can be written as (see \citet{bach2} Proposition 1.2)
\begin{equation}\label{e1}
   \partial \Omega (\omega )=
   \begin{cases}
     \{z \in \mathbb{R}^{p}; \Omega^{*}(z)\leq 1\} & \text{if } \omega=0 \\
     \{z \in \mathbb{R}^{p}; \Omega^{*}(z)= 1 \bigwedge z^{T}w = \Omega(\omega) \} & \text{if } \omega \neq 0.
   \end{cases}
\end{equation}
We are able to apply these properties to our estimator $\bh$. Lemma \ref{l0} implies that 
$$ \bh \text{ is optimal }\Leftrightarrow -\frac{1}{\lambda}\nabla\ltsr{Y-X\bh}\in \partial\Omega(\bh).$$

This means that, in the case $\ltsr{\eh}>0$, for the square root regularization estimator $\bh$ it holds true that
\begin{equation}\label{e2}
   \bh \text{ is optimal } \Leftrightarrow \frac{X^T(Y-X\bh)}{n\lambda \ltsr{Y-X\bh}}\in \partial \Omega (\bh ).
\end{equation}
By combining equation \eqref{e1} with \eqref{e2} we can write the KKT conditions as

\begin{equation}\label{e3}
   \bh \text{ is optimal } \Leftrightarrow 
   \begin{cases}
     \Omega^{*}\left(\frac{\hat{\epsilon}^{T}X}{n\ltsr{\hat{\epsilon}}}\right)\leq \lambda                                    & \text{if } \bh=0 \\
     \Omega^{*}\left(\frac{\hat{\epsilon}^{T}X}{n\ltsr{\hat{\epsilon}}}\right)=    \lambda                                    & \text{if } \bh \neq 0.\\
     \bigwedge \frac{\hat{\epsilon}^{T}X\bh}{n\ltsr{\hat{\epsilon}}} = \lambda \Omega(\bh)                                                                   
   \end{cases}
\end{equation}
What we might first remark about equation \eqref{e3} is that in the case of $\bh \neq 0$ the second part can be written as
$$\hat{\epsilon}^{T}X\bh/n = \Omega(\bh)\cdot \Omega^{*}\left(\frac{\hat{\epsilon}^{T}X}{n}\right).$$
This means that we in fact have equality in the generalized Cauchy-Schwartz Inequality for these two $p-$dimensional vectors.
Furthermore let us remark that the equality
$$\hat{\epsilon}^{T}X\bh/n = \Omega(\bh)\lambda \ltsr{\hat{\epsilon}}$$ 
trivially holds true for the case where $\bh=0$.
It is important to remark here that, in contrast to the KKT conditions for the LASSO, we have an additional $\ltsr{\eh}$ term in the 
expression $ \Omega^{*}\left(\frac{\hat{\epsilon}^{T}X}{n\ltsr{\hat{\epsilon}}}\right)$. This nice scaling leads to the property that the theoretical $\lambda$ is independent of $\sigma$.\\
With the KKT conditions we are able to formulate a generalized type of KKT conditions. This next lemma is needed for the proofs in the next chapter.
\begin{lemma}\label{l00}
 For the square root type estimator $\bh$ we have for any $\beta\in\mathbb{R}^{p}$ and when $\ltsr{\eh}\neq 0$
 $$\frac{1}{\ltsr{\hat{\epsilon}}}\hat{\epsilon}^{T}X(\beta-\bh)/n+\lambda\Omega(\bh)\leq \lambda\Omega(\beta).$$
\end{lemma}
\begin{proof}
 First we need to look at the inequality from the KKT's, which holds in any case
 \begin{equation}\label{ee00}
  \Omega^{*}\left(\frac{\hat{\epsilon}^{T}X}{n\ltsr{\hat{\epsilon}}}\right)\leq \lambda .
 \end{equation}
 And by the definition of the dual norm and the maximum, we have with \eqref{ee00}
 \begin{align*}
  \frac{1}{\ltsr{\hat{\epsilon}}}\hat{\epsilon}^{T}X\beta/n &\leq \Omega(\beta)\cdot\max_{\beta\in\mathbb{R}^{p},\Omega(\beta)\leq 1}\frac{\hat{\epsilon}^{T}}{\ltsr{\hat{\epsilon}}}X\beta/n\\
                                                       &= \Omega(\beta)\cdot\Omega^{*}\left(\frac{\hat{\epsilon}^{T}X}{n\ltsr{\hat{\epsilon}}}\right)\\
                                                       &\leq \Omega(\beta)\lambda.\nrthis \label{ee1}
 \end{align*}

 The second equation from the KKT's, which again holds in any case, is
 \begin{equation}
  \frac{1}{\ltsr{\hat{\epsilon}}}\hat{\epsilon}^{T}X\bh/n = \lambda\Omega(\bh).\label{ee2}
 \end{equation}
 Now putting \eqref{ee1} and \eqref{ee2} together we get the result.

\end{proof}

%

\section{Sharp Oracle Inequalities for the square root regularization estimators}
We provide sharp oracle inequalities for the estimator $\bh$ with a norm $\Omega$ that satisfies a so called weak decomposability condition.
An oracle inequality is a bound on the estimation and prediction errors. This shows how good these estimators are in estimating the parameter vector $\beta^{0}$.
This is an extension of the sharp oracle results given in \citet{sara1} for LASSO type of estimators, which in turn was an generalization of the sharp oracle inequalities for the LASSO and 
nuclear norm penalization in \citet{sharpe1} and \citet{sharpe2}.
Let us first introduce all the necessary definitions and concepts.
Some normed versions of values need to be introduced:
\begin{align*}
 f&=\frac{\lambda\Omega(\bo)}{\ltsr{\epsilon}}\\
 \lambda^{S^c}&=\frac{\Omega^{S^{c}*}((\epsilon^{T}X)_{S^{c}})}{n\ltsr{\epsilon}}\\
 \lambda^{S}&=\frac{\Omega^{*}((\epsilon^{T}X)_{S})}{n\ltsr{\epsilon}}\\
 \lambda^{m}&=\max(\lambda^{S},\lambda^{S^{c}})\\
 \lambda^{0}&=\frac{\Omega^{*}(\epsilon^{T}X)}{n\ltsr{\epsilon}}.
\end{align*}
For example the quantity $f$ gives the measure of the true underlying normalized sparsity. $\Omega^{S^{c}}$ denotes a norm on $\mathbb{R}^{p-|S|}$ which will shortly 
be defined in Assumption II. Furthermore $\lambda^{m}$ will take the role of the theoretical (unknown) $\lambda$. 
If we compare this to the case of the LASSO we see that instead of the $\ell_{\infty}-$norm we generalized it to the dual norm of $\Omega$. 
Also remark that in $\lambda^{m}$ a term $\frac{1}{\ltsr{\epsilon}}$ appears.
This scaling is due to the square root regularization, which will be the reason that $\lambda$ can be chosen independently of the unknown standard deviation $\sigma$.
Now we will give the two main assumptions that need to hold in order to prove the oracle inequalities. Assumption I deals with avoiding overfitting, and the main concern of Assumption II is
that the norm has the desired property of promoting a structured sparse solution $\bh$. We will later see, that the structured sparsity norms in \citet{mi2} and \citet{mi1} are all of this form. 
Thus, Assumption II is quite general.

\textbf{Assumption I (overfitting):}\\
 If $\ltsr{\eh}=0$, then $\bh$ does the same thing as the Ordinary Least Squares (OLS) estimator $\beta_{OLS}$, namely it overfits. That is why we need
 a lower bound on $\ltsr{\eh}$. In order to achieve this lower bound we make the following assumptions:
 $$P\left(Y\in \{\widetilde{Y}: \min_{\beta, \text{s.t.} X\beta=\widetilde{Y}}\Omega(\beta)\leq \ltsr{\epsilon}\}\right)=0.$$
 $$\frac{\lambda^{0}}{\lambda}\big(1+2f\big)< 1.$$
 The $\frac{\lambda^{0}}{\lambda}$ term makes sure that we introduce enough sparsity (no overfitting).


\textbf{Assumption II (weak decomposability):}\\
 Assumption II is fulfilled for a set $S\subset \{1,...,p\}$ and a norm $\Omega$ on $\mathbb{R}^{p}$ 
 if this norm is weakly decomposable, and $S$ is an allowed set for this norm.
 This was used by \citet{sara1} and goes back to \citet{bach2}. It is an assumption on the structure of the sparsity inducing norm.
 By the triangle inequality we have:
 $$\Omega(\beta_{S^{c}})\geq \Omega(\beta)-\Omega(\beta_{S}).$$
 But we will also need to lower bound this by another norm evaluated at $\beta_{S^{c}}$. 
 This is motivated by relaxing the following decomposability property of the $\ell_1$-norm:
 $$\lo{\beta}=\lo{\beta_{S}}+\lo{\beta_{S^{c}}}, \forall \text{ sets } S\subset{1,...,p} \text{ and all }\beta\in \mathbb{R}^{p}.$$
 This decomposability property is used to get oracle inequalities for the LASSO. But we can relax this property, and introduce weakly decomposable norms.
\begin{definition}[Weak decomposability]
 A norm $\Omega$ in $\mathbb{R}^{p}$ is called weakly decomposable for an index set $S\subset\{1,...,p\}$, if
 there exists a norm $\Omega^{S^{c}}$ on $\mathbb{R}^{|S^{c}|}$ such that
 $$\forall \beta\in\mathbb{R}^{p} \ \ \Omega(\beta)\geq \Omega(\beta_{S})+\Omega^{S^{c}}(\beta_{S^{c}}).$$
\end{definition}
 Furthermore we call a set $S$ allowed if $\Omega$ is a weakly decomposable norm for this set.
 
\begin{remark}
 In order to get a good oracle bound, we will choose the norm $\Omega^{S^{c}}$ as large as possible. We will also choose the allowed sets $S$ in such a way to reflect the active set $S_{0}$. Otherwise
 we would of course be able to choose as a trivial example the empty set $S=\varnothing$.
\end{remark}

Now that we have introduced the two main assumptions, we can introduce other definitions and concepts also used in \citet{sara1}.
\begin{definition}
 For $S$ an allowed set of a weakly decomposable norm $\Omega$, and $L>0$ a constant, the $\Omega-$eigenvalue is defined as
 $$\delta_{\Omega}(L,S):=\min\left\{ \|X\beta_{S}-X\beta_{S^{c}}\|_{n}: \Omega(\beta_{S})=1, \Omega^{S^{c}}(\beta_{S^{c}})\leq L\right\}.$$
 Then the $\Omega-$effective sparsity is defined as 
 $$\Gamma_{\Omega}^{2}(L,S):= \frac{1}{\delta_{\Omega}^{2}(L,S)}.$$
\end{definition}
The $\Omega-$eigenvalue is the distance between the two sets (\citet{sara2})
 $\{X\beta_{S}: \Omega(\beta_{S})=1\}$ and $\{X\beta_{S^{c}}: \Omega^{S^{c}}(\beta_{S^{c}})\leq L\}$, see Figure \ref{eff}.
The additional discussion about these definitions will follow after the main theorem. The $\Omega-$eigenvalue generalizes the compatibility constant (\citet{sara4}).

For the proof of the main theorem we need some small lemmas. For any vector $\beta\in \mathbb{R}^{p}$ the $(L,S)-$cone condition for a 
norm $\Omega$ is satisfied if $\Omega^{S^{c}}(\beta_{S^{c}})\leq L\Omega(\beta_{S})$,
with $L>0$ a constant and $S$ an allowed set.

 The proof of Lemma \ref{l2} can be found in \citet{sara1}.
 It shows the connection between the $(L,S)-$cone condition and the $\Omega-$eigenvalue. We bound $\Omega(\beta_{S})$ by a multiple of $\lVert X\beta \rVert_{n}$.
 \begin{lemma}\label{l2}
  Let $S$ be an allowed set of a weakly decomposable norm $\Omega$ and $L>0$ a constant. Then we have that the $\Omega-$eigenvalue is of the following form:
  $$\dO (L,S)=\min \left\lbrace \frac{\lVert X\beta \rVert_{n}}{\Obs} , \beta \text{ satisfies the cone condition and } \beta_{S}\neq 0 \right\rbrace.$$
  We have $\Obs \leq \GO (L,S) \lVert X\beta \rVert_{n}.$
  \end{lemma}

We will also need a lower and an upper bound for $\ltsr{\eh}$, as already mentioned in Assumption I. The next Lemma \ref{l3} gives such bounds.

\begin{lemma}\label{l3}
Suppose that Assumption I holds true.
Then
$$1+f\geq \frac{\ltsr{\eh}}{\ltsr{\epsilon}}\geq \frac{1-\frac{\lambda^{0}}{\lambda}(1+2f)}{f+2}>0.$$

\end{lemma}

\begin{proof}
 The upper bound is obtained by the definition of the estimator
 $$\ltsr{Y-X\bh}+\lambda\Omega(\bh)\leq\ltsr{Y-X\bo}+\lambda\Omega(\bo).$$
 Therefore we get
 $$\ltsr{\eh}\leq \ltsr{\epsilon}+\lambda\Omega(\bo).$$
 Dividing by $\ltsr{\epsilon}$ and by the definition of $f$ we get the desired upper bound.
 The main idea for the lower bound is to use the triangle inequality 
 $$\ltsr{\eh}=\ltsr{\epsilon-X(\bh-\bo)}\geq \ltsr{\epsilon}-\ltsr{X(\bh-\bo)},$$
 and then upper bound $\ltsr{X(\bh-\bo)}$.
 With Lemma \ref{l00} we get an upper bound for $\ltsr{X(\bh-\bo)}$,
 \begin{align*}\ltsr{X(\bh-\bo)}^{2}&\leq \epsilon^{T}X(\bh-\bo)/n+\lambda\ltsr{\eh}(\Omega(\bo)-\Omega(\bh))\\
                                   &\leq \lambda^{0}\ltsr{\epsilon}\Omega(\bh-\bo)+\lambda\ltsr{\eh}(\Omega(\bo)-\Omega(\bh))\\
                                   &\leq \lambda^{0}\ltsr{\epsilon}(\Omega(\bh)+\Omega(\bo))+\lambda\ltsr{\eh}(\Omega(\bo)-\Omega(\bh))\\
                                   &\leq \lambda^{0}\ltsr{\epsilon}\Omega(\bh)+\Omega(\bo)(\lambda^{0}\ltsr{\epsilon}+\lambda\ltsr{\eh}).
 \end{align*}
 In the second line we used the definition of the dual norm, and the Cauchy-Schwartz inequality.
 Again by the definition of the estimator we have
 \begin{align*}
  \Omega(\bh)\leq\frac{\ltsr{\epsilon} }{\lambda}+\Omega(\bo).
 \end{align*}
 And we are left with
 $$\ltsr{X(\bh-\bo)}    \leq  \ltsr{\epsilon}\sqrt{\frac{\lambda^{0}}{\lambda}\left(1
    +2\frac{\lambda\Omega(\beta^{0})}{\ltsr{\epsilon}}+\frac{\lambda}{\lambda^{0}}\cdot\frac{\ltsr{\eh}}{\ltsr{\epsilon}}\cdot\frac{\lambda\Omega(\beta^{0})}{\ltsr{\epsilon}}\right)}.$$
 By the definition of $f$ we get
 \begin{align*}\ltsr{X(\bh-\bo)}    &\leq  \ltsr{\epsilon}\sqrt{\frac{\lambda^{0}}{\lambda}\left(1
    +2f+\frac{\lambda}{\lambda^{0}}\frac{\ltsr{\eh}}{\ltsr{\epsilon}}f\right)}. \\
                       &=     \ltsr{\epsilon}\sqrt{\frac{\lambda^{0}}{\lambda}+2\frac{\lambda^{0}}{\lambda}f+\frac{\ltsr{\eh}}{\ltsr{\epsilon}}f}.  
 \end{align*}
 
 Now we get
 \begin{align}
  \ltsr{\eh}&\geq \ltsr{\epsilon}-\ltsr{X(\bh-\bo)} \nonumber\\
            &\geq \ltsr{\epsilon}-\ltsr{\epsilon}\sqrt{\frac{\lambda^{0}}{\lambda}+2\frac{\lambda^{0}}{\lambda}f+\frac{\ltsr{\eh}}{\ltsr{\epsilon}}f}\label{eq-lp}
 \end{align}
 Let us rearrange equation \eqref{eq-lp} further in the case $\frac{\ltsr{\eh}}{\ltsr{\epsilon}}<1$
  $$\frac{\lambda^{0}}{\lambda}+2\frac{\lambda^{0}}{\lambda}f+\frac{\ltsr{\eh}}{\ltsr{\epsilon}}f\geq \left(1-\frac{\ltsr{\eh}}{\ltsr{\epsilon}}\right)^{2}$$
  $$\frac{\ltsr{\eh}}{\ltsr{\epsilon}}f\geq 1-2\frac{\ltsr{\eh}}{\ltsr{\epsilon}}+\frac{\ltsr{\eh}^{2}}{\ltsr{\epsilon}^{2}}-\frac{\lambda^{0}}{\lambda}(1+2f)$$
  $$\frac{\ltsr{\eh}}{\ltsr{\epsilon}}f+2\frac{\ltsr{\eh}}{\ltsr{\epsilon}}\geq 1-\frac{\lambda^{0}}{\lambda}(1+2f)$$
  $$\frac{\ltsr{\eh}}{\ltsr{\epsilon}}\geq \frac{1-\frac{\lambda^{0}}{\lambda}(1+2f)}{f+2}\overset{\text{Assumption I}}{>}0.$$
 On the other hand if $\frac{\ltsr{\eh}}{\ltsr{\epsilon}}>1$, we already get a lower bound which is bigger than $\frac{1-\frac{\lambda^{0}}{\lambda}(1+2f)}{f+2}$.

\end{proof}

Finally we are able to present the main theorem. This theorem gives sharp oracle inequalities on the prediction error expressed in the $\ell_2$-norm, 
and the estimation error expressed in the $\Omega$ and $\Omega^{S^{c}}$ norms.
\begin{remark}\label{rem:rem9}Let us first briefly remark that in the Theorem \ref{th2} we need to assure that $\lambda^{*}-\lambda^{m}>0$. The assumption 
$\frac{\lambda^{m}}{\lambda}< 1/a$, with $a$ chosen as in Theorem \ref{th2}, together with the fact that $\lambda^{0}\leq \lambda^{m}$ leads to the desired inequality
$$\frac{\lambda^{*}}{\lambda}=\frac{1-\frac{\lambda^{0}}{\lambda}(1+2f)}{f+2}\geq \frac{1-\frac{\lambda^{m}}{\lambda}(1+2f)}{f+2}>\frac{\lambda^{m}}{\lambda}.$$
\end{remark}

\begin{theorem}\label{th2}
 Assume that  $0\leq \delta< 1$, and also that $a\lambda^{m}< \lambda$, with the constant 
 $a=3(1+f).$
 We invoke also Assumption I (overfitting) and Assumption II (weak decomposability) for $S$ and $\Omega$. Here the allowed set $S$ is chosen
 such that the active set $S_{\beta}:=\supp (\beta)$ is a subset of $S$.
 Then it holds true that\\
 \begin{align}\label{eq:theoremm}
   &\ltn{X(\bh-\bo )}+2\delta\ltsr{\epsilon}\left[(\lambda^{*}+\lambda^{m})\Omega(\bh_{S}-\beta)+(\lambda^{*}-\lambda^{m})\Omega^{S^{c}}(\bh_{S^{c}})\right]\nonumber  \\ 
   &\leq \ltn{X(\beta-\bo )}+ \ltsr{\epsilon}^{2}\left[ (1+\delta)(\tilde{\lambda}+\lambda^{m})\right]^{2} \GO^{2} (L_{S},S),
 \end{align}

 with $L_{S}:= \frac{\tilde{\lambda}+\lambda^{m}}{\lambda^{*}-\lambda^{m}}\frac{1+\delta}{1-\delta}$ and
 
 \begin{align*}
   \lambda^{*} &:=\lambda\left(\frac{1-\frac{\lambda^{0}}{\lambda}(1+2f)}{f+2}\right), &&\tilde{\lambda}\quad:=\lambda(1+f).
 \end{align*}
 
 Furthermore we get the two oracle inequalities
 \begin{align*}
  \ltn{X(\bh-\bo )}&\leq \ltn{X(\beta_{\star}-\bo )}\\
  &\quad+\ltsr{\epsilon}^{2}(1+\delta)^{2}(\tilde{\lambda}+\lambda^{S_{\star}^{c}})^2 \cdot \GO^{2} (L_{S_{\star}},S_{\star})\\
  \Omega(\bh_{S_{\star}}-\beta_{\star})+\Omega^{S_{\star}^{c}}(\bh_{S_{\star}^{c}})&\leq \frac{1}{2\delta\ltsr{\epsilon}}\cdot \frac{\ltn{X(\beta_{\star}-\bo )}}{\lambda^{*}-\lambda^{m}}+...\\
    &\quad +\frac{(1+\delta)^{2}\ltsr{\epsilon}}{2\delta }\cdot \frac{(\tilde{\lambda}+\lambda^{m})^2}{\lambda^{*}-\lambda^{m}} \cdot \GO^{2}(L_{S_{\star}},S_{\star}).
 \end{align*}
 
For all fixed allowed sets $S$ define
$$\beta_{\star}(S):=\amin\limits_{\beta:\text{ }\supp(\beta)\subseteq S}\left( \ltn{X(\beta-\bo )}+ \ltsr{\epsilon}^{2}\left[ (1+\delta)(\tilde{\lambda}+\lambda^{m})\right]^{2} \GO^{2} (L_{S},S)\right).$$
Then $S_{\star}$ is defined as
\begin{small}\begin{align}S_{\star}&:=\amin\limits_{S \text{ allowed}}\left(\ltn{X(\beta_{\star}(S)-\bo )}
                        + \ltsr{\epsilon}^{2}\left[ (1+\delta)(\tilde{\lambda}+\lambda^{m})\right]^{2} \GO^{2} (L_{S},S)\right),\label{e3q}\\
\beta_{\star}&:= \beta_{\star}(S_{\star})\label{e2q}\end{align}\end{small}
it attains the minimal right hand side of the oracle inequality \eqref{eq:theoremm}.
An improtant special case of equation \eqref{eq:theoremm} is to choose $\beta\equiv \beta^{0}$ with $S\supseteq S_{0}$ allowed.
The term $\ltn{X(\beta-\bo )}$ vanishes in this case and only the $\Omega-$effective sparsity term remains for the upper bound. But it is not obvious in which cases and 
whether $\beta_{\star}$ leads to a substantially lower bound than $\beta^{0}$.
\end{theorem}

\begin{proof}
 Let $\beta \in \Rp \text{ and let } S \text{ be an allowed set containing the active set of }\beta$. We need to distinguish 2 cases. The second case is the more substantial one.\\
 $\text{\underline{Case 1:} }$ Assume that\\
 $$\slt{X(\bh-\bo)}{X(\bh-\beta)} \leq -\dell.$$
 Here $\slt{u}{v}:=v^{T}u/n$, for any two vectors $u,v\in \mathbb{R}^{n}$ .
 In this case we can simply use the following calculations to verify the theorem.

   \begin{align*}
     &\ltn{X(\bh-\bo)}-\ltn{X(\beta-\bo)} +...\\
     &\qquad+2\dell\\
     &\qquad\qquad= 2\slt{X(\bh-\bo)}{X(\bh-\beta)}-\ltn{X(\beta-\bh)}\\
     &\qquad\qquad\quad +2\dell\\
     &\qquad\qquad \leq -\ltn{X(\beta-\bh)}\\
     &\qquad\qquad \leq 0
   \end{align*}

 Now we can turn to the more important case.\\
 $\text{\underline{Case 2:} }$ Assume that\\
 $$\slt{X(\bh-\bo)}{X(\bh-\beta)} \geq -\dell.$$
 We can reformulate Lemma \ref{l00} with $Y-X\bh=X(\bo-\bh)+\epsilon$, then we get:
 $$\frac{\slt{X(\bh-\bo)}{X(\bh-\beta)}}{\ltsr{\eh}}+\lambda\Omega(\bh)\leq \frac{\slt{\epsilon}{X(\bh-\beta)}}{\ltsr{\eh}}+\lambda\Omega(\beta).$$
 This is equivalent to
 \begin{equation}\label{uno2}\slt{X(\bh-\bo)}{X(\bh-\beta)}+\ltsr{\eh}\lambda\Omega(\bh)\leq \slt{\epsilon}{X(\bh-\beta)}+\ltsr{\eh}\lambda\Omega(\beta).\end{equation}
 By the definition of the dual norm and the generalized Cauchy-Schwartz inequality we have
 \begin{align*}\slt{\epsilon}{X(\bh-\beta)}&\leq \ltsr{\epsilon}\left(\lambda^{S}\Omega(\bh_{S}-\beta)+\lambda^{S^{c}}\Omega^{S^{c}}(\bh_{S^{c}})\right)\\
                                           &\leq \ltsr{\epsilon}\left(\lambda^{m}\Omega(\bh_{S}-\beta)+\lambda^{m}\Omega^{S^{c}}(\bh_{S^{c}})\right)
 \end{align*}
 
 Inserting this inequality into \eqref{uno2} we get
 \begin{align}\slt{X(\bh-\bo)}{X(\bh-\beta)}+\ltsr{\eh}\lambda\Omega(\bh)&\leq\ltsr{\epsilon}\left(\lambda^{m}\Omega(\bh_{S}-\beta)+\lambda^{m}\Omega^{S^{c}}(\bh_{S^{c}})\right)\nonumber\\
                                                           &\quad +\ltsr{\eh}\lambda\Omega(\beta).\label{uno4}\end{align}
 Then by the weak decomposability and the triangle inequality in \eqref{uno4}
 
 $$\slt{X(\bh-\bo)}{X(\bh-\beta)}+\ltsr{\eh}\lambda\left(\Omega(\bh_{S})+\Omega^{S^{c}}(\bh_{S^{c}})\right)$$
 \begin{equation}\leq \ltsr{\epsilon}\left(\lambda^{m}\Omega(\bh_{S}-\beta)+\lambda^{m}\Omega^{S^{c}}(\bh_{S^{c}})\right)+\ltsr{\eh}\lambda\left(\Omega(\bh_{S}-\beta)+\Omega(\bh_{S})\right).\label{uno5}\end{equation}
 By inserting the assumption of case 2
 $$\slt{X(\bh-\bo)}{X(\bh-\beta)} \geq -\dell, $$
 into \eqref{uno5} we get\begin{small}
 $$\Big(\lambda\ltsr{\eh}-\lambda^{m}\ltsr{\epsilon}-\delta\ltsr{\epsilon}(\lambda^{*}-\lambda^{m})\Big)\Omega^{S^{c}}(\bh_{S^{c}})\leq \left(\lambda\ltsr{\eh}+\lambda^{m}\ltsr{\epsilon}
 +\delta\ltsr{\epsilon}(\tilde{\lambda}+\lambda^{m})\right)\Omega(\bh_{S}-\beta).$$\end{small}
 By assumption $a\lambda^{m}<\lambda$ we have that $\lambda^{*}>\lambda^{m}$ (see Remark \ref{rem:rem9}) and therefore
 $$\Omega^{S^{c}}(\bh_{S^{c}})\leq \left(\frac{\tilde{\lambda}+\lambda^{m}}{\lambda^{*}-\lambda^{m}} \right){\cdot}\frac{1+\delta}{1-\delta}\cdot\Omega(\bh_{S}-\beta).$$
 We have applied Lemma \ref{l3} in the last step, in order to replace the estimate $\ltsr{\eh}$ with $\ltsr{\epsilon}$. By the definition of $L_{S}$ we have
 \begin{equation}\Omega^{S^{c}}(\bh_{S^{c}})\leq L_{S}\Omega(\bh_{S}-\beta).\label{uno6}\end{equation}
  Therefore with Lemma \ref{l2} we get
  \begin{equation}\Omega(\bh_{S}-\beta)\leq \GO(L_{S},S)\ltsrr{X(\bh-\beta)}.\label{uno7}\end{equation}
 
 Inserting \eqref{uno7} into \eqref{uno5}, together with Lemma \ref{l3} and $\delta<1$, we get
 \begin{align*}&\slt{X(\bh-\bo)}{X(\bh-\beta)}+\delta\ltsr{\epsilon}(\lambda^{*}-\lambda^{m})\Omega^{S^{c}}(\bh_{S^{c}}) \\
                                            &\leq (1+\delta-\delta)\ltsr{\epsilon}(\lambda\ltsr{\eh}/\ltsr{\epsilon}+\lambda^{m})\Omega(\bh_{S}-\beta)\\
                                            &\leq (1+\delta)\ltsr{\epsilon}(\tilde{\lambda}+\lambda^{m})\GO(L_{S},S)\ltsrr{X(\bh-\beta)}-\delta\ltsr{\epsilon}(\lambda^{*}+\lambda^{m})\Omega(\bh_{S}-\beta)
 \end{align*}
 Because $\forall u,v \in \mathbb{R}, \text{ }0\leq(u-v)^{2}$ it holds true that $uv\leq 1/2(u^{2}+v^{2})$.

 Therefore with $a=(1+\delta)\ltsr{\epsilon}(\tilde{\lambda}+\lambda^{m})\GO(L_{S},S)$ and $b= \ltsrr{X(\bh-\beta)}$ we have
 \begin{align*}
 &\slt{X(\bh-\bo)}{X(\bh-\beta)}+\delta\ltsr{\epsilon}(\lambda^{*}-\lambda^{m})\Omega(\bh_{S^{c}})^{S^{c}}+\delta\ltsr{\epsilon}(\lambda^{*}+\lambda^{m})\Omega(\bh_{S}-\beta) \\
 &\leq \frac{1}{2}(1+\delta)^{2}\ltsr{\epsilon}^{2}(\tilde{\lambda}+\lambda^{m})^{2}\GO^{2}(L_{S},S)+\frac{1}{2}\ltn{X(\bh-\beta)}. 
 \end{align*}
 Since
 $$2\slt{X(\bh-\bo)}{X(\bh-\beta)}=\ltn{X(\bh-\beta^{0})}-\ltn{X(\beta-\beta^{0})}+\ltn{X(\bh-\beta)},$$
 we get
 \begin{align}\ltn{X(\bh-\beta^{0})}&+2\delta\ltsr{\epsilon}\left((\lambda^{*}-\lambda^{m})\Omega(\bh_{S^{c}})^{S^{c}}+(\lambda^{*}+\lambda^{m})\Omega(\bh_{S}-\beta)\right)\nonumber\\
&\leq  (1+\delta)^{2}\ltsr{\epsilon}^{2}(\tilde{\lambda}+\lambda^{m})^{2}\GO^{2}(L_{S},S)+\ltn{X(\beta-\beta^{0})}.\label{eq:total}\end{align}
 This gives the sharp oracle inequality. The two oracle inequalities mentioned are just a split up version of inequality \eqref{eq:total}, where for the second oracle inequality we need
 to see that $\lambda^{*}-\lambda^{m}\leq \lambda^{*}+\lambda^{m}$.
\end{proof}

Remark that the sharpness in the oracle inequality of Theorem \ref{th2} is the constant one in front of the term $\ltn{X(\beta-\bo )}$. 
Because we measure a vector on $S_{\star}$ by $\Omega$ and on the inactive set $S_{\star}^{c}$ by the norm $\Omega^{S^{c}}$, we take here $\Omega(\bh_{S_{\star}}-\beta_{\star})$ and $\Omega^{S_{\star}^{c}}(\bh_{S_{\star}^{c}})$
as estimation errors.\\
If we choose $\lambda$ of the same order as $\lambda^{m}$ (i.e. $a\lambda=\lambda^{m}$, with $a>0$ a constant), then we can simplify the oracle inequalities. 
This is comparable to the oracle inequalities for the LASSO, see for example \citet{laso1}, \citet{laso2}, \citet{laso3}, \citet{sara4} and further references can be found in \citet{sara3}.
\begin{corollary}\label{th3}
Take $\lambda$ of the order of $\lambda^{m}$ (i.e. $\lambda^{m}=C\lambda$, with 
$0<C<\frac{1}{3(f+1)}$ a constant). 
Invoke the same assumptions as in Theorem \ref{th2}. Here we also use the same notation of an optimal $\beta_{\star}$ with $S_{\star}$ as in equation \eqref{e3q} and \eqref{e2q}.
Then we have
 \begin{align*}
  \ltn{X(\bh-\bo )}&\leq  \ltn{X(\beta_{\star}-\bo )}+C_{1}\lambda^2 \cdot \GO^{2} (L_{S_{\star}},S_{\star})\\
  \Omega(\bh_{S_{\star}}-\beta_{\star})+\Omega^{S_{\star}^{c}}(\bh_{S_{\star}^{c}})&\leq C_{2}\left(\frac{\ltn{X(\beta_{\star}-\bo )}}{\lambda} + C_{1}\lambda \cdot \GO^{2} (L_{S_{\star}},S_{\star})\right).
 \end{align*}
Here $C_{1}$ and $C_{2}$ are the constants:
\begin{align}
\label{eq:c1}C_{1}&:=(1+\delta)^{2}\cdot \ltsr{\epsilon}^{2}(f+C+1)^{2},\\
\label{eq:c2}C_{2}&:=\frac{1}{2\delta\ltsr{\epsilon}}\cdot \frac{1}{\sqrt{1-2C(1+2f)}-C}.\end{align}
\end{corollary}

First let us explain some of the parts of Theorem \ref{th2} in more detail. We can also study what happens to the bound if we additionally assume Gaussian errors, see Proposition \ref{proposition1}.\\

\textbf{On the two parts of the oracle bound:}\\
The oracle bound is a trade-off between two parts, which we will discuss now.
Let us first remember that if we set $\beta\equiv\beta^{0}$ in the sharp oracle bound, only the term with the $\Omega-$effective sparsity will not vanish on the right hand side of the bound. 
But due to the minimization over $\beta$ in the definition of $\beta_{\star}$ we might even do better than that bound.\\
The \underline{first part} consisting of minimizing $\ltn{X(\beta-\bo )}$ can be thought of the error made due to approximation, hence we call it the approximation error.
If we fix the support $S$, which can be thought of being determined by the second part, then minimizing $\ltn{X(\beta-\bo )}$ is 
just a projection onto the subspace spanned by $S$, see Figure \ref{projection}. So if $S$ has a similar structure
than the true unknown support $S_{0}$ of $\beta^{0}$, this will be small.\\

\begin{figure}[H]
\centering
  \includegraphics[width = 8cm]{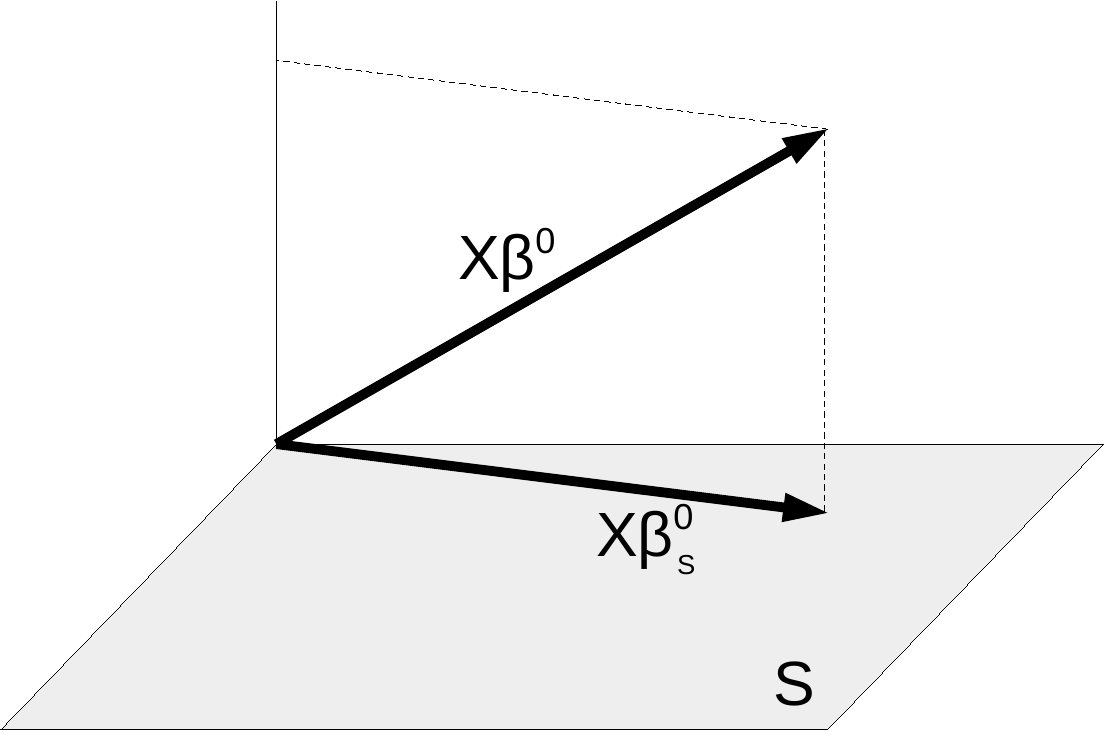}
\caption{approximation error}
\label{projection}
\end{figure}

The \underline{second part} containing $\GO^{2} (L_{S},S)$ is due to estimation errors. There, minimizing over $\beta$ will affect the set $S$.
We have already mentioned that. It is one over the squared distance between the two sets
$\{X\beta_{S}: \Omega(\beta_{S})=1\}$ and $\{X\beta_{S^{c}}: \Omega^{S^{c}}(\beta_{S^{c}})\leq L\}$. Figure \ref{eff} shows this distance.
This means that if the vectors in $X_{S}$ and $X_{S^{c}}$ show
a high correlation the distance will shrink and the $\Omega-$effective sparsity will blow up, which we try to avoid. This distance depends also on the two chosen sparsity norms
$\Omega$ and $\Omega^{S^{c}}$. It is crucial to choose norms that reflect the true underlying sparsity in order to get a good bound. Also the constant $L_{S}$ should be small.\\

\begin{figure}[H]
\centering
  \includegraphics[width = 6cm]{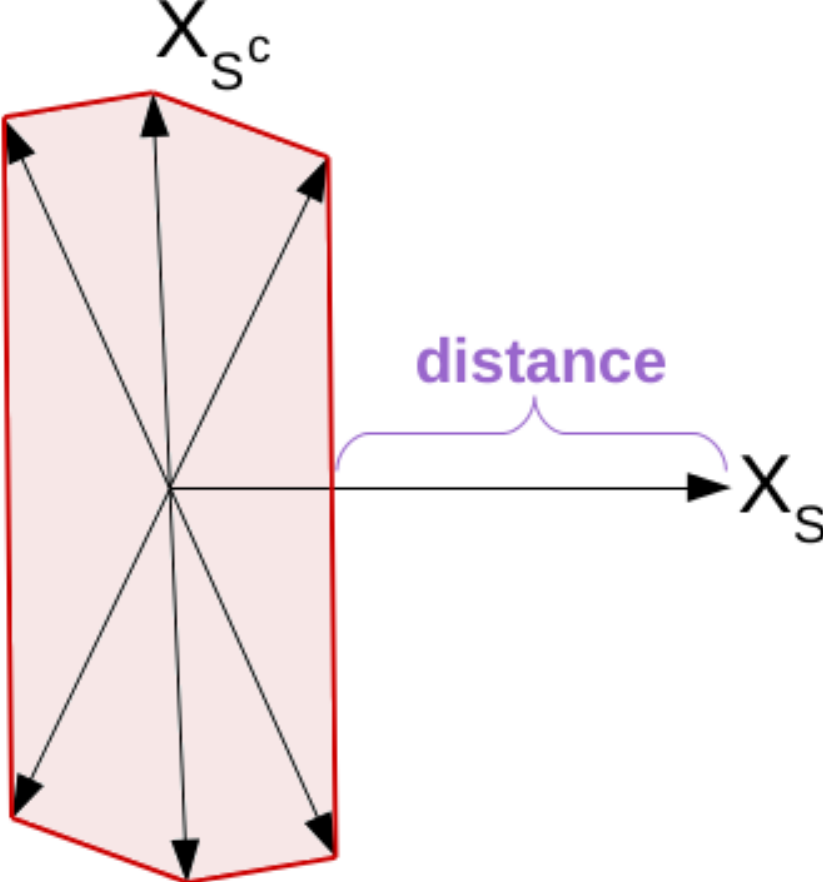}
\caption{The $\Omega$-eigenvalue}
\label{eff}
\end{figure}


\textbf{On the randomness of the oracle bound:}\\
Until now, the bound still contains some random parts, for example in $\lambda^{m}$. In order to get rid of that random part we need to introduce the following sets
$$\mathcal{T}:= \left\{\max\left(\frac{\Omega^{*}((\epsilon^{T}X)_{W})}{n\ltsr{\epsilon}}, \frac{\Omega^{W^{c}*}((\epsilon^{T}X)_{W^{c}})}{n\ltsr{\epsilon}}\right)\leq d \right\},$$
$\text{ where } d\in \mathbb{R}, \text{ and any allowed set }W.$
We need to choose the constant $d$ in such a way, that we have a high probability for this set. In other words we try to bound the random part by a non random constant with a very high probability.
In order to do this we need some assumptions on the errors. Here we assume Gaussian errors. Let us also remark that $\frac{\Omega^{*}((\epsilon^{T}X)_{W})}{n\ltsr{\epsilon}}$ is 
normalized by $\ltsr{\epsilon}$.
This normalization occurs due to the special form of the Karush-Kuhn-Tucker conditions. Thus the square root of the residual sum of
squared errors is responsible for this normalization. In fact, this normalization is the main reason why $\lambda$ does not contain the unknown variance. So the square root part of the estimator
makes the estimator pivotal.
Now in the case of Gaussian errors, we can use the concentration inequality from Theorem 5.8 in \citet{lugosi} and get the following proposition.
Define first:
\begin{table}[H]
\centering
\label{my-label}
\begin{tabular}{llll}
 $Z_{1}:=$&$\frac{\Omega^{*}((\epsilon^{T}X)_{W})}{n\ltsr{\epsilon}}$&$V_{1}:=$&$Z_{1}\ltsr{\epsilon}/\sigma$\\
 $Z_{2}:=$&$\frac{\Omega^{W^{c}*}((\epsilon^{T}X)_{W^{c}})}{n\ltsr{\epsilon}}$&$V_{2}:=$&$Z_{2}\ltsr{\epsilon}/\sigma$   \\
 $Z\text{ }:=$&$\max(Z_{1},Z_{2})$&$V\text{ }:=$& $\max(V_{1},V_{2})$
\end{tabular}
\end{table}
\begin{proposition}\label{proposition1}
  Suppose that we have $i.i.d.$ Gaussian errors $\epsilon\sim\mathcal{N}(0,\sigma^{2}I)$, and that the following normalization $(X^{T}X/n)_{i,i}=1, \forall i\in \{1,...,p\}$ holds true.
  Let $B:=\{z\in\mathbb{R}^{p}: \Omega(z)\leq 1\}$ be the unit $\Omega-$ball, and $B_{2}:=\sup_{b\in B}b^{T}b$ an $\Omega-$ball and $\ell_{2}-$ball comparison.
  Then we have for all $d>\E V$ and $\Delta>1$
  $$\Prob (\mathcal{T})\geq 1-2e^{-\frac{(d-\E V)^2\Delta^{2}}{2 B_{2}/n}}-2e^{-\frac{n}{4}(1-\Delta^{2})^{2}}.$$
\end{proposition}

\begin{proof}
Let us define $\Sigma^{2}:=\sup\limits_{b\in B}\E{\left(\frac{(\epsilon^{T}X)_{W}b}{n\sigma}\right)^{2}}$ and calculate it
 \begin{align*}
  \Sigma^{2}&=\sup_{b\in B}\Var\left(\frac{(\epsilon^{T}X)_{W}b}{n\sigma}\right)\\
  &=\sup_{b\in B}\Var\left(\sum_{w\in W}\sum_{i =1}^{n}\frac{\epsilon_{i}}{n}X_{wi}b_{w}\right)\frac{1}{\sigma^{2}}\\
  &=\sup_{b\in B}\left(\sum_{w\in W}b_{w}^{2}\sum_{i =1}^{n}X_{wi}^{2}\Var\left(\frac{\epsilon_{i}}{n}\right)\right)\frac{1}{\sigma^{2}}\\
  &=\sup_{b\in B}b_{W}^{T}\cdot b_{W}/n \sum_{i =1}^{n}X_{wi}^{2}/n\\
  &=\sup_{b\in B}b_{W}^{T}\cdot b_{W}/n \leq B_{2}/n.\tageq\label{e:sil}
 \end{align*}
 These calculations hold true as well for $W^{c}$ instead of $W$. Furthermore in the subsequent inequalities we can subsitute $W$ with $W^{c}$ and use $Z_{2},V_{2}$ instead of $Z_{1},V_{1}$ to 
 get an analogous result.
 We have $\frac{(\epsilon^{T}X)_{W}b}{\sigma n} \sim\mathcal{N}(0,b_{W}^{2}/n)$. This is an almost surely continuous centred Gaussian process. 
 Therefore we can apply Theorem 5.8 from \citet{lugosi}
 \begin{equation}\label{e:lug}\Prob(V_{1}-\E V_{1}\geq c)\leq e^{-\frac{c^2}{2B_{2}/n}}.\end{equation}
 Now to get to a probability inequality for $Z_{1}$ we use the following calculations
 \begin{align*}
  \Prob\left(Z_{1}-\E V_{1} \geq d\right)&\leq \Prob\left(\frac{V_{1}\sigma}{\ltsr{\epsilon}}-\E V_{1} \geq d\wedge \ltsr{\epsilon}>\sigma \Delta \right)+\Prob(\ltsr{\epsilon}\leq \sigma\Delta)   \\ 
                                               &\leq \Prob\left(V_{1}-\E V_{1} \Delta  > d\Delta  \right)+\Prob(\ltsr{\epsilon}\leq \sigma\Delta)\\
                                               &\leq \Prob\left(V_{1}-\E V_{1}  > d\Delta  \right)+\Prob(\ltsr{\epsilon}\leq \sigma\Delta)\\
                                               &\leq e^{-\frac{d^2\Delta^{2}}{2 B_{2}/n}} +\Prob(\ltsr{\epsilon}\leq \sigma\Delta).\tageq\label{e:sil3} 
 \end{align*}
 The calculations above use the union bound and that a bigger set containing another set has a bigger probability. 
 Furthermore we have applied equations \eqref{e:sil} and \eqref{e:lug}.
 Now we are left to give a bound on $\Prob(\ltsr{\epsilon}/\sigma \leq \Delta)$. For this we use the corollary to Lemma 1 from \citet{laurent} together 
 with the fact that $\ltsr{\epsilon}/\sigma=\sqrt{R/n}$ with $R=\sum_{i=1}^{n}(\epsilon_{i}/\sigma)^{2}\sim\chi^{2}(n)$. We obtain
 \begin{align*}
  \Prob\left(R\leq n-2\sqrt{nx}\right)&\leq \exp(-x)\\
  \Prob\left(\sqrt{\frac{R}{n}}\leq \sqrt{1-2\sqrt{\frac{x}{n}}}\right)&\leq \exp(-x)\\
  \Prob\left(\frac{\ltsr{\epsilon}}{\sigma}\leq \Delta\right)&\leq e^{-\frac{n}{4}(1-\Delta^{2})^{2}}.\tageq\label{e:sil2} 
 \end{align*}

 Combining equations \eqref{e:sil3} and \eqref{e:sil2} finishes the proof:
 \begin{align*}
  \Prob(\mathcal{T})&=\Prob(\max(Z_{1},Z_{2})\leq d)\\
  &=\Prob(Z_{1}\leq d\cap Z_{2}\leq d) \\
  &\geq \Prob(Z_{1}\leq d)+\Prob(Z_{2}\leq d)-1\\
  &\geq 1-\Prob(Z_{1}\geq d)-\Prob(Z_{2}\geq d)\\
  &\geq 1-2e^{-\frac{(d-\E V)^2\Delta^{2}}{2 B_{2}/n}}-2e^{-\frac{n}{4}(1-\Delta^{2})^{2}}.
 \end{align*}
\end{proof}

So the probability that the event $\mathcal{T}$ does not occur decays exponentially. This is what we mean by having a very high probability.
Therefore we can take $d=t\cdot \sqrt{\frac{\frac{2}{n}B_{2}}{\Delta^{2}}}+\E\left[ V\right]$ with $\Delta^{2}=1-t\frac{2}{\sqrt{n}}$, 
where $t=\sqrt{\log\left(\frac{4}{\alpha}\right)}$ and $2 e^{-n/2}<\alpha$ to ensure $\Delta^{2}>0$. With this we get
\begin{equation}\label{eq:proob}\Prob (\mathcal{T})\geq 1-\alpha.\end{equation}
First remark that the term $\frac{\epsilon^{T}}{\sigma}$ is now of the right scaling, because $\epsilon_{i}/\sigma\sim \mathcal{N}(0,1)$. This is the whole point of the square root regularization.

Here $B_{2}$ can be thought of comparing the $\Omega-$ball in direction $W$ to the $\ell_{2}-$ball in direction $W$, because if the norm $\Omega$ is the $\ell_{2}-$norm, then $B_{2}=1$. 
Moreover, for every norm there exists a constant $D$ such that for all $\beta$  it holds
$$\lVert \beta \rVert_{2} \leq D\Omega(\beta).$$
Therefore the $B_{2}$ of $\Omega$ satisfies 
$$B_{2}\leq D^{2}\sup_{b\in B}\Omega(b_{W})^{2}\leq D^{2}.$$
Thus we can take
$$\boxed{d=t\cdot \frac{D}{\Delta}\sqrt{\frac{2}{n}}+\E\left[V\right]}$$
$$\boxed{\Delta^{2}=1-t\sqrt{\frac{2}{n}}, \text{ with }t=\sqrt{\log\left(\frac{4}{\alpha}\right)}}.$$
What is left to be determined is $\E\left[V\right]$. In many cases we can use a adjusted version of the main theorem in \citet{mi3} for Gaussian complexities to obtain this expectation. 
All the examples below can be calculated in this way.
So, in the case of Gaussian errors, we have the following new version of Corollary \ref{th3}.
\begin{corollary}
 Take $\lambda= t/\Delta \cdot D\sqrt{\frac{2}{n}}+\E\left[V\right]$, where $t,\delta,V$ and $D$ are defined as above. 
 Invoke the same assumptions as in Theorem \ref{th2} and additionally assume Gaussian errors. Use the notation from Corollary \ref{th3}.
 Then with probability $1-\alpha$ the following oracle inequalities hold true
 \begin{align*}
  \ltn{X(\bh-\bo )}&\leq  \ltn{X(\beta_{\star}-\bo )}+C_{1}\lambda^2 \cdot \GO^{2} (L_{S_{\star}},S_{\star}) \\
  \Omega(\bh_{S_{\star}}-\beta_{\star})+\Omega^{S_{\star}^{c}}(\bh_{S_{\star}^{c}})&\leq C_{2}\left(\frac{\ltn{X(\beta_{\star}-\bo )}}{\lambda} + C_{1}\lambda \cdot \GO^{2} (L_{S_{\star}},S_{\star})\right).
 \end{align*}\label{coro1}
\end{corollary}

Now we still have a $\ltsr{\epsilon}^{2}$ term in the constants \eqref{eq:c1}, \eqref{eq:c2} of the oracle inequality. In order to handle this we need Lemma 1 from \citet{laurent}.
Which translates in our case to the probability inequality
$$\Prob \left(\ltsr{\epsilon}^{2}\leq \sigma^{2}\left(1+2x+2x^{2}\right)\right)\geq 1-\exp\left(-n\cdot x^{2}\right).$$
Here $x>0$ is a constant. Therefore we have that $\ltsr{\epsilon}^{2}$ is of the order of $\sigma^{2}$ with exponentially decaying probability in $n$. We could also write this in the following form
$$\Prob \left(\ltsr{\epsilon}^{2}\leq \sigma^{2}\cdot C \right)\geq 1-\exp\left(-\frac{n}{2}\left(C-\sqrt{2C-1}\right)\right).$$
Here we can choose any constant $C>1$ big enough and take the bound $\sigma^2\cdot C$ for $\ltsr{\epsilon}^{2}$ in the oracle inequality.
A similar bounds can be found in \citet{laurent} for $1/\ltsr{\epsilon}^{2}$. This takes care of the random part in the sharp oracle bound with the Gaussian errors.
\begin{figure}[H]
\centering
\subfloat[$\ell_{1}$-norm]{\includegraphics[width = 5.4cm]{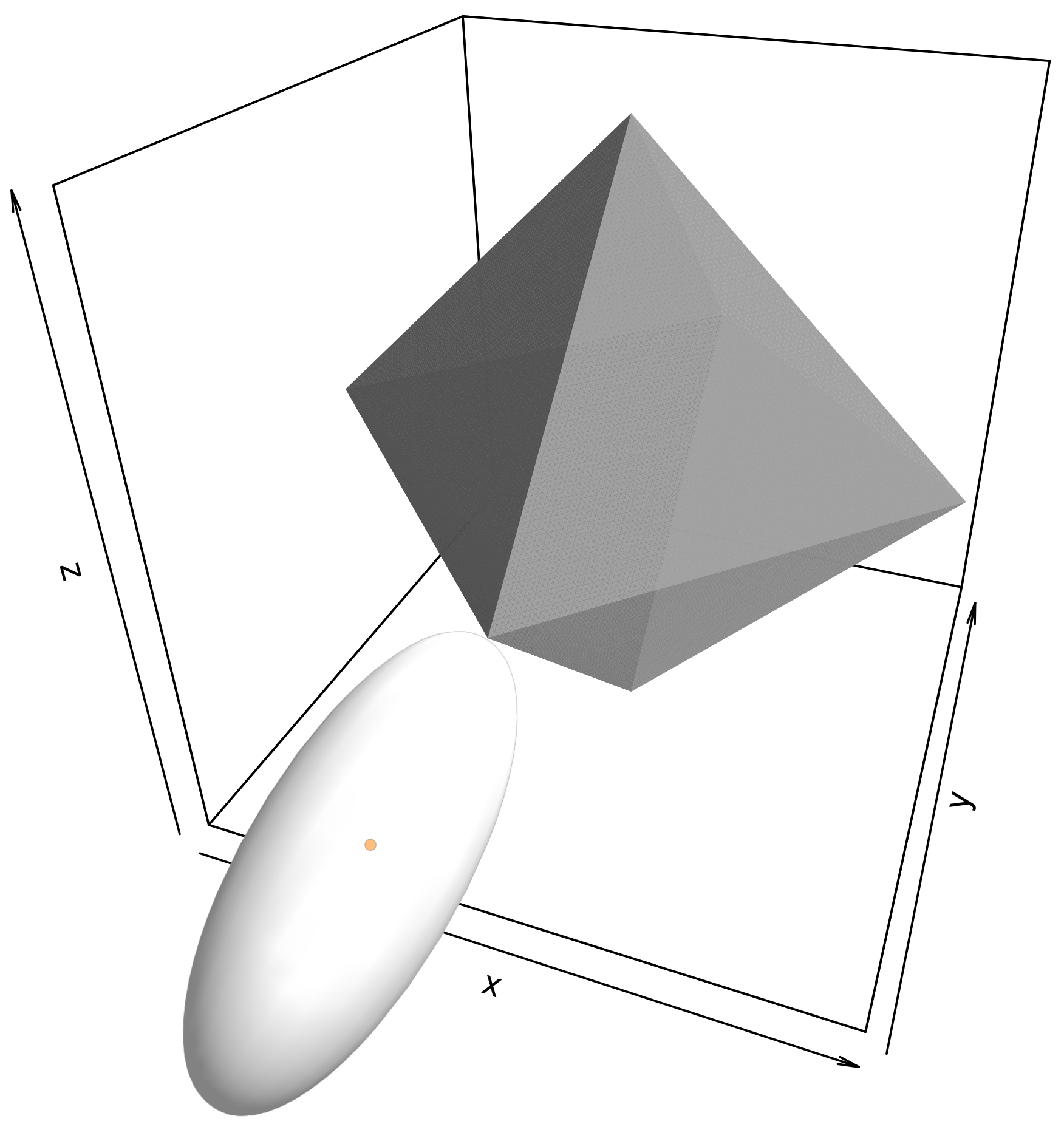}}\qquad
\subfloat[Group Lasso norm with groups $\{x\},\{y,z\}$]{\includegraphics[width = 5.4cm]{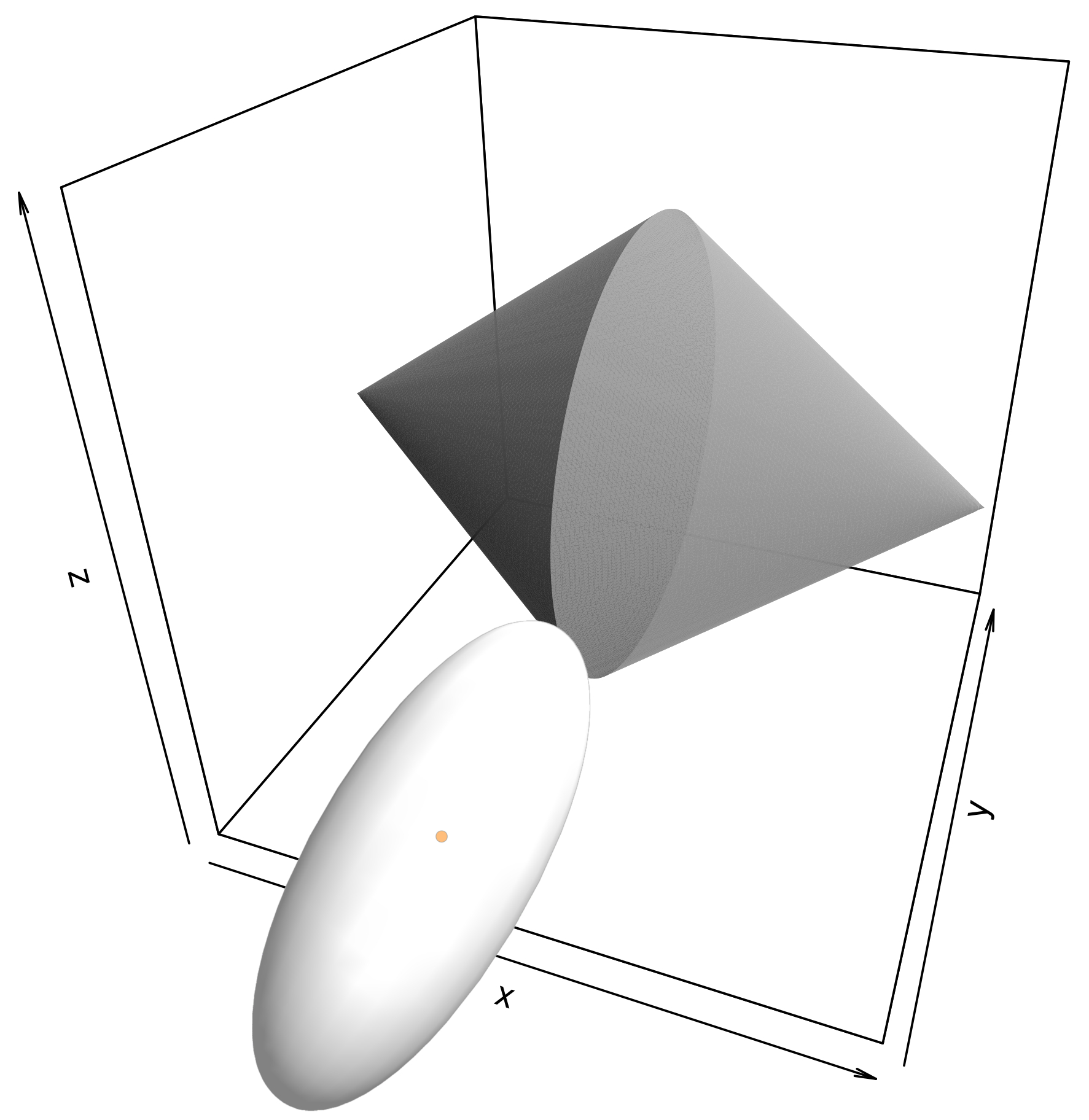}}\\
\subfloat[sorted $\ell_{1}$-norm with a $\lambda$ sequence $1>0.5>0.3$]{\includegraphics[width = 5.4cm]{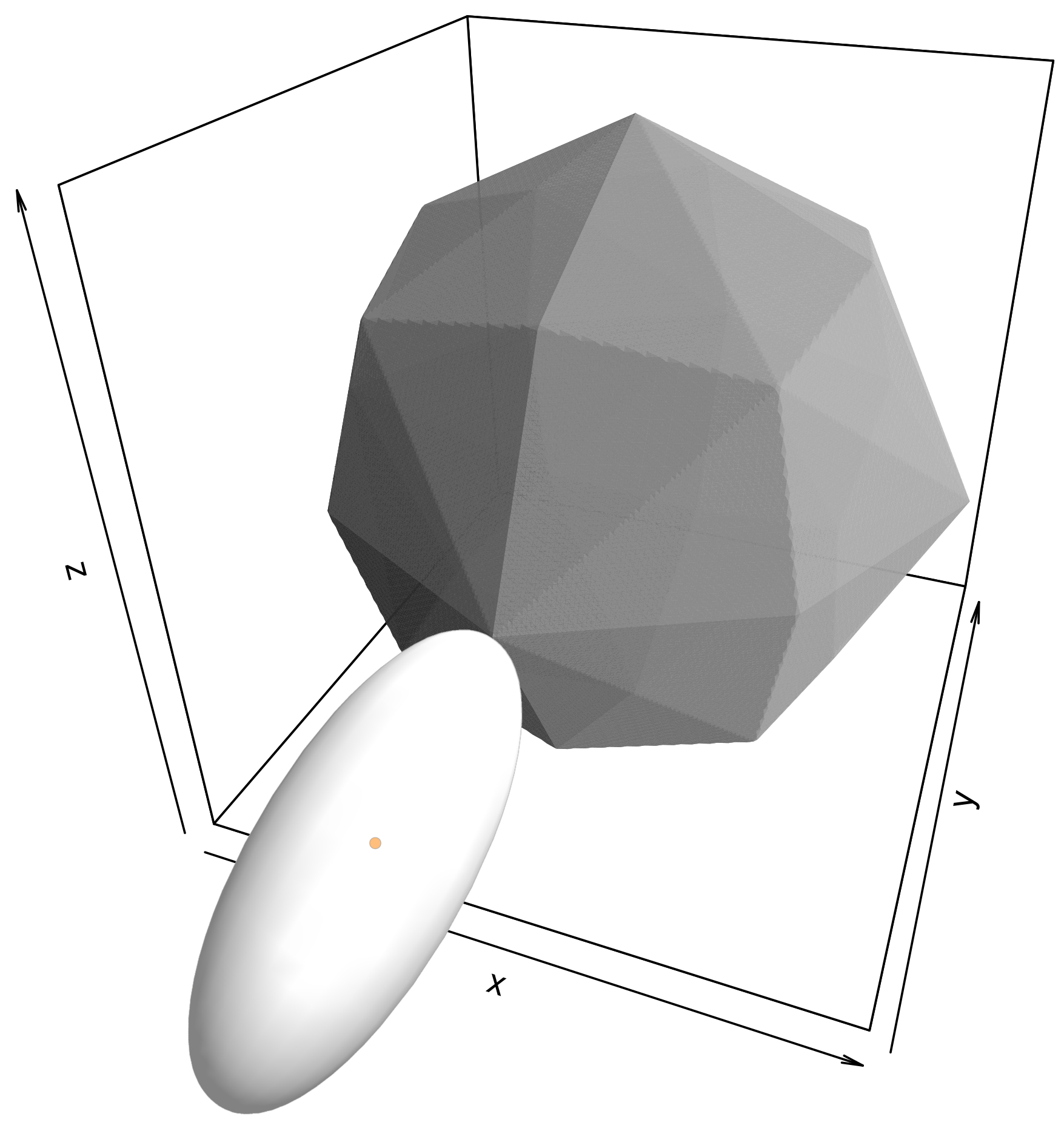}}\qquad
\subfloat[wedge norm]{\includegraphics[width = 5.4cm]{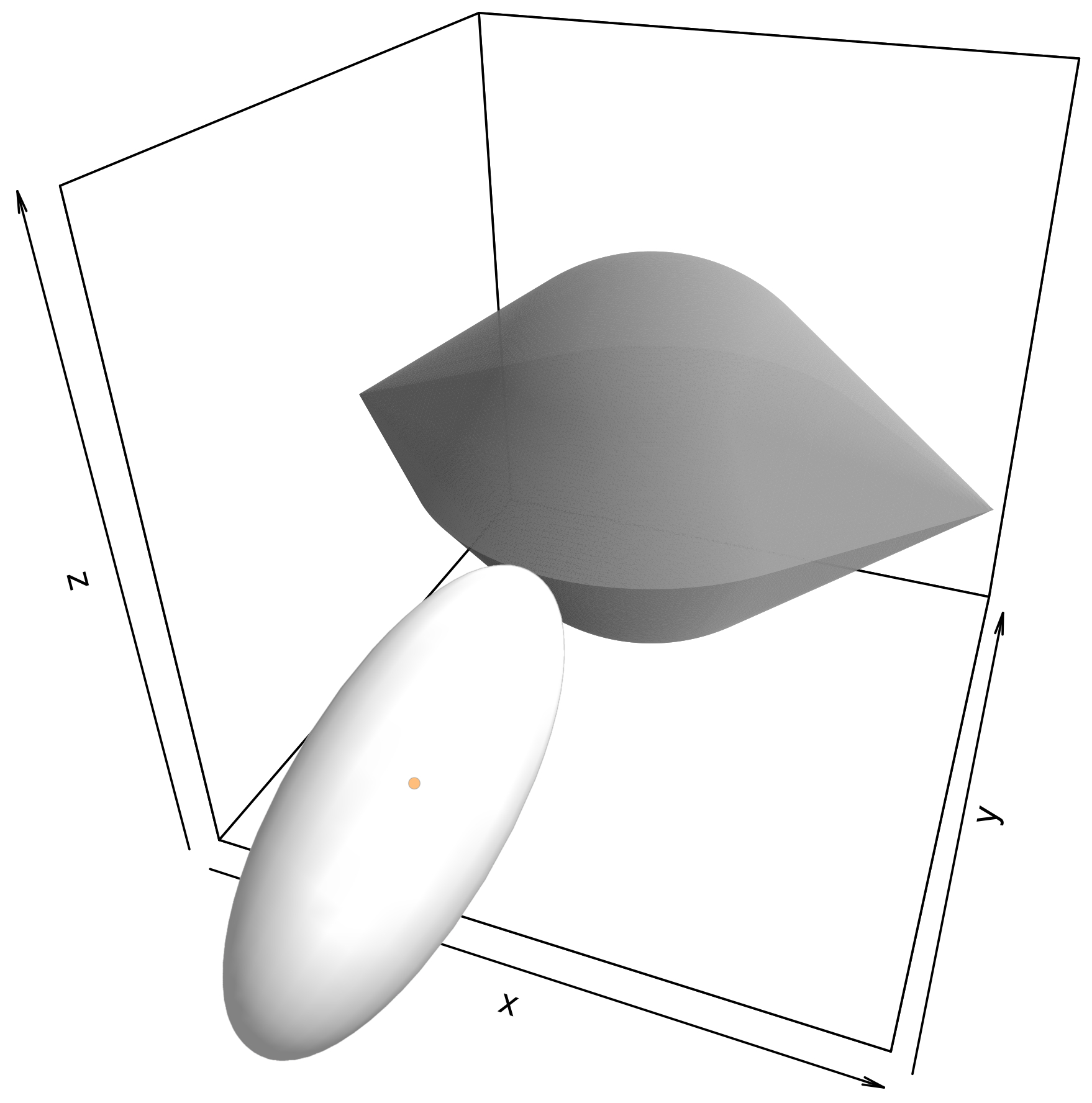}}\\
\caption{Pictorial description of how the estimator $\bh$ works, with unit balls of different sparsity inducing norms.}
\label{exs}
\end{figure}

\newpage
\section{Examples}
Here we will give some examples of estimators where our sharp oracle inequalities hold. Figure \ref{exs} shows the unit balls of some sparsity inducing norms that we will use as examples. In order
to give the theoretical $\lambda$ for these examples we will again assume Gaussian errors. Theorem \ref{th2} still holds for all the examples even for non Gaussian errors. 
Some of the examples will introduce new estimators inspired by methods similar to the square root LASSO.

\subsection*{Square Root LASSO}
First we examine the square root LASSO,
$$ \hat \beta_{srL} := \amin_{\beta \in \mathbb{R}^{p} } \biggl \{ \| Y - X \beta \|_n + \lambda \| \beta \|_1 \biggr \} . $$
Here we use the $\ell_1-$norm as a sparsity measure. We know that the $\ell_1-$norm has the nice property to be able to set certain unimportant parameters individually to zero.
As already mentioned the $\ell_1-$norm has the following decomposability property for any set $S$
$$\lo{\beta}=\lo{\beta_{S}}+\lo{\beta_{S^{c}}},\forall \beta \in \mathbb{R}^{p}.$$
Therefore we also have weak decomposability for all subsets $S\subset\{1,...,p\}$ with $\Omega^{S^{c}}$ being the $\ell_1-$norm again. Thus Assumption II is fulfilled for all sets $S$ and so
we are able to apply Theorem \ref{th2}.\\
Furthermore for the square root LASSO we have that $D = 1$. This is because the $\ell_{2}-$norm is bounded by the $\ell_1-$norm without any constant.
So in order to get the value of $\lambda$ we need to calculate the expectation of the dual norm of $\frac{\epsilon^{T}X}{\sigma n}$.
The dual norm of $\ell_1$ is the $\ell_{\infty}-$norm. By \citet{mi3}, we also have
$$\max\left(\E \left[\frac{\left\lVert(\epsilon^{T}X)_{S_{\star}^{c}}\right\rVert_{\infty}}{n\sigma }\right],\E \left[\frac{\left\lVert(\epsilon^{T}X)_{S_{\star}}\right\rVert_{\infty}}{n\sigma }\right]\right)\leq \sqrt{\frac{2}{n}}\left(2+\sqrt{\log(|p|)}\right).$$
Therefore the theoretical $\lambda$ for the square root LASSO can be chosen as
$$\boxed{\lambda= \sqrt{\frac{2}{n}}\left(t/\Delta+2+\sqrt{\log(|p|)}\right)}.$$
Even though this theoretical $\lambda$ is very close to being optimal, it is not optimal, see for example \citet{sara5}.
In the special case of the $\ell_{1}-$norm penalization, we can simplify Corollary \ref{coro1}:
\begin{corollary}[Square Root LASSO]
 Take $\lambda=  \sqrt{\frac{2}{n}}\left(t/\Delta+2+\sqrt{\log(|p|)}\right),$ where $t>0$ and $\Delta>1$ are chosen as in \eqref{eq:proob}. Invoke the same assumptions as in Corollary \ref{coro1}.
 Then for $\Omega(\cdot)=\lo{\cdot}$, we have with probability $1-\alpha$ that the following oracle inequalities hold true:
 \begin{align*}
  \ltn{X(\bh_{srL}-\bo )}&\leq  \ltn{X(\beta_{\star}-\bo )}+C_{1}\lambda^2 \cdot \GO^{2} (L_{S_{\star}},S_{\star}) \\
  \lo{\bh_{srL}-\beta_{\star}}&\leq C_{2}\left(\frac{\ltn{X(\beta_{\star}-\bo )}}{\lambda} + C_{1}\lambda \cdot \GO^{2} (L_{S_{\star}},S_{\star})\right).
 \end{align*}\label{coroe1}
\end{corollary}
Remark that in Corollary \ref{coroe1} we have an oracle inequality for the estimation error $\lo{\bh_{srL}-\beta_{\star}}$ in $\ell_{1}$. 
This is due to the decomposability of the $\ell_1-$norm. In other examples we will have the sum of two norms.

\subsection*{Group Square Root LASSO}
 In order to set groups of variables simultaneously to zero, and not only individual variables, we will look at a different sparsity inducing norm. 
 Namely a $\ell_1-$type norm for grouped variables, called the group LASSO norm. The group square root LASSO was introduced by \citet{led1} as
 $$\bh_{gsrL}:=\amin_{\beta\in\mathbb{R}^{p}}\left\{ \ltsr{Y-X\beta}+\lambda \sum_{j=1}^{g}\sqrt{|G_{j}|}\lVert \beta_{G_{j}}\rVert_{2}\right\}.$$
 Here $g$ is the total number of groups, and $G_{j}$ is the set of variables that are in the $j$th group.
 Of course the $\ell_1-$norm is a special case of the group LASSO norm, when $G_{j}=\{j\}$ and $g=p$.
 
 The group LASSO penalty is also weakly decomposable with $\Omega^{S^{c}}=\Omega$, for any $S=\bigcup\limits_{j \in \mathcal{J}}G_{j}$, with any $\mathcal{J} \subset \{1,...,g\}.$
 So here the sparsity structure of the group LASSO norm induces the sets $S$ to be of the same sparsity structure in order to fulfil Assumption II.
 Therefore the Theorem \ref{th2} can also be applied in this case.\\
 How do we need to choose the theoretical $\lambda$? For the group LASSO norm we have $B_{2}\leq 1$.
 One can see this due to the fact that $\sqrt{a_{1}}+...+\sqrt{a_{g}}\geq \sqrt{a_{1}+...+a_{g}}$ for $g$ positive constants. And also $|G_{j}|\geq 1$ for all groups.
 Therefore
 $$\sum_{j=1}^{g}\sqrt{|G_{j}|}\lVert \beta_{G_{j}}\rVert_{2}\geq \sqrt{\sum_{i=1}^{p}\beta_{i}^{2}}.$$
 Remark that the dual norm is $\Omega^{*}(\beta)=\max\limits_{1\leq j\leq g}\lVert \beta_{G_{j}}\rVert_{2}/\sqrt{|G_{j}|} $. With \citet{mi3} we have
 
 $$\max\left(\E \left[\frac{\Omega^{*}\left((\epsilon^{T}X)_{S_{\star}^{c}}\right)}{n\sigma}\right],\E \left[\frac{\Omega^{*}\left((\epsilon^{T}X)_{S_{\star}}\right)}{n\sigma}\right]\right)\leq \sqrt{\frac{2}{n}}\left(2+\sqrt{\log(g)}\right).$$
 That is why $\lambda$ can be taken of the following form
 $$\boxed{\lambda= \sqrt{\frac{2}{n}}\left(t/\Delta+2+\sqrt{\log(g)}\right)}.$$
 And we get a similar corollary for the group square root LASSO like the Corollary \ref{coroe1} for the square root LASSO.
 In the case of the group LASSO, there are better results for the theoretical penalty level available, see for example Theorem 8.1 in \citet{sara3}. 
 This takes the minimal group size into account.

\subsection*{Square Root SLOPE}
 Here we introduce a new method called the square root SLOPE estimator, which is also part of the square root regularization family.
 Let us thus take a look at the sorted $\ell_1$ norm with some decreasing sequence $\lambda_{1}\geq\lambda_{2}\geq...\geq\lambda_{p}> 0$ ,
 $$J_{\lambda}(\beta):=\lambda_{1}\lvert\beta\rvert_{(1)}+...+\lambda_{p}\lvert\beta\rvert_{(p)}.$$
 This was shown to be a norm by \citet{zeng}.

Let $\pi$ be a permutation of $\{ 1 , \ldots , p \} $. The identity permutation is
denoted by $id$. In order to show weak decomposability for the norm $J_{\lambda}$ we need the following lemmas.

\begin{lemma}[Rearrangement Inequality] Let $\beta_1 \ge \cdots \ge \beta_p$ be a decreasing sequence of non-negative numbers.
The sum $\sum_{i=1}^p \lambda_i \beta_{\pi(i) } $ is maximized over all
permutations $\pi$ at $\pi = id $.
\end{lemma}

{\bf Proof.} The result is obvious when $p=2$. Suppose now that it is true for sequences of
length $p-1$. We then prove it for sequences of length $p$ as follows. Let $\pi$ be an arbitrary 
permutation with $j:= \pi(p)$. Then
$$ \sum_{i=1}^p \lambda_i \beta_{\pi(i)} = \sum_{i=1}^{p-1} \lambda_i \beta_{\pi(i)} + \lambda_p \beta_j .$$
By induction 
\begin{align*}\sum_{i=1}^{p-1} \lambda_i \beta_{\pi(i)} &\le \sum_{i=1}^{j-1} \lambda_i \beta_i +\sum_{i=j+1}^p \lambda_{i-1} \beta_i \\
&= \sum_{i\not=j } \lambda_i \beta_i + \sum_{i=j+1}^p ( \lambda_{i-1} - \lambda_i) \beta_i\\
&= \sum_{i=1}^p \lambda_i \beta_i + \sum_{i=j+1}^p ( \lambda_{i-1} - \lambda_{i}) \beta_i - \lambda_j \beta_j .
\end{align*}

Hence we have
\begin{align*} \sum_{i=1}^p \lambda_i \beta_{\pi(i)}  &\le \sum_{i=1}^p \lambda_i \beta_i +
\sum_{i=j+1}^p ( \lambda_{i-1} - \lambda_i) \beta_i + (\lambda_j  - \lambda_p)\beta_j \\
& =  \sum_{i=1}^p \lambda_i \beta_i + \sum_{i=j+1}^p ( \lambda_{i-1} - \lambda_i) \beta_i -
 \sum_{i=j+1}^p ( \lambda_{i-1} - \lambda_i) \beta_j\\
& =  \sum_{i=1}^p \lambda_i \beta_i + \sum_{i=j+1}^p ( \lambda_{i-1} - \lambda_i) (\beta_i - \beta_j) .
\end{align*}
 Since $\lambda_{i-1} \ge \lambda_i$ for all $1\le i \le p $ (defining $\lambda_0 =0$) and $\beta_i \le \beta_j$ for all
 $i>j$ we know that
 $$ \sum_{i=j+1}^p ( \lambda_{i-1} - \lambda_i) (\beta_i - \beta_j) \le 0 . $$
 \hfill $\sqcup \mkern -12mu \sqcap$

\begin{lemma} Let 
$$ \Omega (\beta) = \sum_{i=1}^p \lambda_i | \beta |_{(i)} , $$
and 
$$ \Omega^{S^{c}} ( \beta_{S^{c}}) = \sum_{l=1}^{r} \lambda_{p-r+l} |\beta |_{(l,S^{c})}, $$
where $r=p-s$ and
$ |\beta |_{(1,S^{c})} \ge \cdots \ge| \beta |_{(r,S^{c})}$ is the ordered sequence
in $\beta_{S^{c}}$. 
Then $\Omega (\beta) \ge \Omega(\beta_S) + \Omega^{S^{c}} (\beta_{S^{c}})$. 
Moreover $\Omega^{S^{c}}$ is the strongest norm among all $\underline \Omega^{S^{c}}$
for which $\Omega (\beta) \ge \Omega(\beta_S) + \underline \Omega^{S^{c}} (\beta_{S^{c}})$

\end{lemma}

{\bf Proof.} Without loss of generality assume $\beta_1 \ge \cdots \ge \beta_p \ge 0 $.
We have
$$ \Omega (\beta_S) + \Omega^{S^{c}} (\beta_{S^{c}}) = \sum_{i=1}^p \lambda_i \beta_{\pi(i)} $$
for a suitable permutation $\pi$.
It follows that 
$$\Omega (\beta_S) + \Omega^{S^{c}} (\beta_{S^{c}})  \le \Omega (\beta) . $$
To show $\Omega^{S^{c}}$ is the strongest norm it is clear we need only to search among candidates of the form
$$ \underline \Omega^{S^{c}} (\beta_{S^{c}} )= \sum_{l =1}^r \underline \lambda_{p-r+l} \beta_ {\pi^{S^{c}} (l)}$$
where $\{ \underline \lambda_{p-r+l} \}$ is a decreasing positive sequence and
where $\pi^{S^{c}} (1) , \ldots , \pi^{S^{c}} (r) $ is a permutation of indices in $S^{c}$.
 
This is then maximized  by ordering the indices in $S^c$ in decreasing order. 
But then it follows that the largest norm is obtained by taking $\underline \lambda_{p-r+l} =
\lambda_{p-r+l} $ for all $l=1 , \ldots , r $. 
\hfill $\sqcup \mkern -12mu \sqcap$

 The SLOPE was introduced by \citet{candes1} in order to better control the false discovery rate, and is defined as:
 $$\hat{\beta}_{SLOPE}:=\arg\min_{\beta\in\mathbb{R}^{p}}\left\{\ltsr{Y-X\beta}^{2}+\lambda J_{\lambda}(\beta)\right\}.$$
 Now we are able to look at the square root SLOPE, which is the estimator of the form:
 $$\hat{\beta}_{srSLOPE}:=\arg\min_{\beta\in\mathbb{R}^{p}}\left\{\ltsr{Y-X\beta}+\lambda J_{\lambda}(\beta)\right\}.$$
 The square root SLOPE replaces the squared $\ell_{2}-$norm with a $\ell_{2}-$norm.
 With Theorem \ref{th2} we have provided a sharp oracle inequality for this new estimator, the square root SLOPE.

 For the SLOPE penalty we have $B_{2}\leq \frac{1}{\lambda_{p}}$, if $\lambda_{p}>0$. This is because
 \begin{align*}
  \frac{J_{\lambda}(\beta)}{\lambda_{p}}&= \frac{\lambda_{1}}{\lambda_{p}}|\beta|_{(1)}+...+\frac{\lambda_{p}}{\lambda_{p}}|\beta|_{(p)}\\
  &\geq \sum_{i=1}^{p}|\beta_{i}| = \lo{\beta}\\
  &\geq \lVert \beta \rVert_{2}.
 \end{align*}
 So the bound gets scaled by the smallest $\lambda$.
 The dual norm of the SLOPE is by Lemma 1 of \citet{zeng2}
 $$J_{\lambda}^{*}(\beta)=\max_{k=1,...,p}\left\{ \bigg(\sum_{j=1}^{k}\lambda_{j}\bigg)^{-1}\cdot \lo{\beta^{(k)}} \right\},$$
 Here $\beta^{(k)}:=(\beta_{(1)},...,\beta_{(k)})^{T}$ is the vector which contains the $k$ largest elements of $\beta$.
 
 Again by \citet{mi3} we have
 $$\max\left(\E \left[\frac{J_{\lambda}^{*}\left((\epsilon^{T}X)_{S_{\star}}\right)}{n\sigma }\right],\left[\frac{J_{\lambda}^{S_{\star}^{c}*}\left((\epsilon^{T}X)_{S_{\star}^{c}}\right)}{n\sigma }\right]\right)\leq \sqrt{\frac{2}{n}}\left(\frac{2\sqrt{2}+1}{\sqrt{2}}+\sqrt{\log(|R^{2}|)}\right).$$
 Here we denote by $R^{2}:=\sum\limits_{i}\frac{1}{\lambda_{i}^{2}}$. Therefore we can choose $\lambda$ as
 $$\boxed{\lambda=\sqrt{\frac{2}{n}}\left(\frac{t}{\lambda_{p}\Delta}+\frac{2\sqrt{2}+1}{\sqrt{2}}+\sqrt{\log(|R^{2}|)}\right)}.$$
 Let us remark that the asymptotic minimaxity of SLOPE can be found in \citet{candes2}.

\subsection*{Sparse Group Square Root LASSO}

 The sparse group square root LASSO can be defined similarly to the sparse group LASSO, see \citet{simon}. This new method is defined as:
 $$\hat{\beta}_{srSGLASSO}:=\amin_{\beta \in \mathbb{R}^{p}} \left\{ \ltsr{Y-X\beta }+\lambda \left\lVert \beta \right\rVert_{1} +\eta \sum_{t=1}^{T}\left\lVert  \beta_{I_{t}} \right\rVert_{2}\sqrt{\lvert G_{t} \rvert} \right\},$$
 where we have a partition as follows, $G_{t}\subset \{1,...,p\} \ \forall t \in {1,...,T}$ , $\bigcup\limits_{t=1}^{T}G_{t}=\{1,...,p\}$ and $G_{i}\cap G_{j}=\varnothing\ \forall i\neq j$.
 This penalty is again a norm and it not only chooses sparse groups by the group LASSO penalty, but also sparsity inside of the groups with the $\ell_1-$norm.
 Define $R(\beta):=  \lambda \left\lVert \beta \right\rVert_{1} + \eta \sum_{t=1}^{T}\left\lVert  \beta_{I_{t}} \right\rVert_{2}\sqrt{\lvert G_{t} \rvert}$ and 
 $R^{S^{c}}(\beta):= \lambda \left\lVert \beta \right\rVert_{1}$. Then we have weak decomposability for any set $S$
 $$R(\beta_{S})+R^{S^{c}}(\beta_{S^{c}})\leq R(\beta).$$
 This is due to the weak decomposability property of the $\ell_1-$norm and 
 $\left\lVert \beta_{S} \right\rVert_{2}= \sqrt{\sum\limits_{j \in S}\beta_{j}^{2}}\leq \sqrt{\sum\limits_{j \in S}\beta_{j}^{2}+\sum\limits_{j \in S^{c}}\beta_{j}^{2}}=\left\lVert \beta\right\rVert_{2}$.
 Now in order to get the theoretical $\lambda$ let us note that if we sum two norms, it is again a norm. Then the dual of this added norm is, because of the supremum taken over the unit ball,
 smaller than dual norm of each one of the two norms individually. So we can invoke the same theoretical $\lambda$ as with the square root LASSO
 $$\boxed{\lambda= \sqrt{\frac{2}{n}}\left(t/\Delta+2+\sqrt{\log(|p|)}\right)}.$$
 And also the theoretical $\eta$ like the group square root LASSO
 $$\boxed{\eta= \sqrt{\frac{2}{n}}\left(t/\Delta+2+\sqrt{\log(g)}\right)}.$$
 But of course we will not get the same Corollary, because the $\Omega-$effective sparsity will be different.

\subsection*{Structured Sparsity}

 Here we will look at the very general concept of structured sparsity norms.
 Let $\mathcal{A}\subset \left[0,\infty\right)^{p}$ be a convex cone such that $\mathcal{A}\cap\left(0,\infty\right)^{p}\neq \varnothing$. Then
 $$\Omega(\beta)=\Omega(\beta;\mathcal{A}):= \min_{a\in \mathcal{A}}\frac{1}{2}\sum\limits_{j=1}^{p}\left(\frac{\beta_{j}^{2}}{a_{j}}+a_{j}\right),$$
 is a norm by \citet{mi2}. Some special cases are for example the $\ell_1-$norm or the wedge or box norm.
 Define
 $$\mathcal{A}_{S}:=\{a_{S}: a\in \mathcal{A}\}.$$
 Then \citet{sara1} also showed that for any $\mathcal{A}_{S}\subset\mathcal{A}$ we have that the set $S$ is allowed and we have weak decomposability for the norm
 $\Omega(\beta)$ with $\Omega^{S^{c}}(\beta_{S^{c}}):=\Omega(\beta_{S^{c}},\mathcal{A}_{S^{c}})$.
 Hence the estimator
 $$\hat{\beta}_{s}=\amin_{\beta \in \mathbb{R}^{p}} \left\{ \ltsr{Y-X\beta }+\lambda \min_{a\in \mathcal{A}}\frac{1}{2}\sum\limits_{j=1}^{p}\left(\frac{\beta_{j}^{2}}{a_{j}}+a_{j}\right)\right\},$$
 has also the sharp oracle inequality.
 The dual norm is given by
 \begin{align*}
 \Omega^{*}(\omega;\mathcal{A})&=\max\limits_{a\in\mathcal{A}(1)}\sqrt{\sum_{j=1}^{p}a_{j}\omega_{j}^{2}}, \text{ }\text{ }\text{ }\text{ }\omega \in \mathbb{R}^{p},\\
 \Omega^{S^{c}*}(\omega;\mathcal{A}_{S^{c}})&=\max\limits_{a\in\mathcal{A}_{S^{c}}(1)}\sqrt{\sum_{j=1}^{p}a_{j}\omega_{j}^{2}}, \text{ }\omega \in \mathbb{R}^{p}.
 \end{align*}
 Here $\mathcal{A}_{S^{c}}(1):=\{a\in\mathcal{A}_{S^{c}}:\lo{a}=1\}$ and $\mathcal{A}(1):=\{a\in\mathcal{A}:\lo{a}=1\}$.
 Then once again by \citet{mi3} we have
   \begin{small}$$\max\left(\E \left[\frac{\Omega^{*}\left((\epsilon^{T}X)_{S_{\star}};\mathcal{A}_{S_{\star}}\right)}{n\sigma}\right], \E \left[\frac{\Omega^{*}\left((\epsilon^{T}X)_{S_{\star}^{c}};\mathcal{A}_{S_{\star}^{c}}\right)}{n\sigma}\right]\right)\leq 
   \sqrt{\frac{2}{n}}\widetilde{\mathcal{A}}_{S_\star}\left(2+\sqrt{\log(|\E(\mathcal{A})|)}\right).$$
   \end{small}
 Here $\E(\mathcal{A})$ are the extreme points of the closure of the set $\left\{\frac{a}{\lo{a}}:a\in\mathcal{A}\right\}$. 
 With the definition $\widetilde{\mathcal{A}}_{S_\star}:=
 \max\left(\sqrt{\sum_{i=1}^{n}\Omega(X_{i,S_{\star}};\mathcal{A}_{S_{\star}})},\sqrt{\sum_{i=1}^{n}\Omega(X_{i,S_{\star}^{c}};\mathcal{A}_{S_{\star}^{c}})}\right)$.
 That is why $\lambda$ can be taken of the following form
 $$\boxed{\lambda= \sqrt{\frac{2}{n}}\left(tD/\Delta+\widetilde{\mathcal{A}}_{S_{\star}}\left(2+\sqrt{\log(|\E(\mathcal{A})|)}\right) \right)}.$$
 Since we do not know $S_{\star}$ we can either upper bound $\widetilde{\mathcal{A}}_{S_\star}$ for a given norm, or use the fact that 
 $\Omega(\beta)\geq \lo{\beta}$ and $\Omega^{*}(\beta)\leq \linf{\beta}$ for all $\beta\in \mathbb{R}^{p}$. Therefore use the same $\lambda$ as for the square root LASSO.
 And we get similar corollaries for the structured sparsity norms like the Corollary \ref{coroe1} for the square root LASSO.

\section{Simulation: Comparison between srLASSO and srSLOPE}
The goal of this simulation is to see how the estimation and prediction errors for the square root LASSO and the square root SLOPE behave under
some Gaussian designs. We propose Algorithm \ref{algo1} to solve the square root SLOPE:
\IncMargin{1em}
\begin{algorithm}\label{algo1}
  \SetAlgoLined
  \caption{srSLOPE}
  \SetKwInOut{Input}{input}
  \SetKwInOut{Output}{output}

\Indm  
  \Input{$\beta^{0}$\hspace{2mm} a starting parameter vector,\\ $\lambda$\hspace{4mm} a desired penalty level with a decreasing sequence,\\ $Y$\hspace{3mm} the response vector,\\ $X$\hspace{3mm} the design matrix.}
  \Output{$\hat{\beta}_{srSLOPE}=\amin\limits_{\beta\in\mathbb{R}^{p}}\left(\ltsr{Y-X\beta}+\lambda J_{\lambda}(\beta)\right)$}
\Indp
  \BlankLine
  \For{$i\leftarrow 0$ \KwTo $i_{\text{stop}}$}{ 
    $\sigma_{i+1}\leftarrow \ltsr{Y-X\beta_{i}}$\;
    $\beta_{i+1} \leftarrow \amin\limits_{\beta\in\mathbb{R}^{p}}\left(\ltsr{Y-X\beta}^{2}+\sigma_{i+1} \lambda J_{\lambda}(\beta)\right)$ \; }
    
\end{algorithm}
\DecMargin{1em}

Note that in Algorithm \ref{algo1} Line 3 we need to solve the usual SLOPE. To solve the SLOPE we have used the algorithm provided in \citet{candes1}. 
For the square root LASSO we have used the R-Package flare by \citet{flare}.

We consider a high-dimensional linear regression model:
$$Y=X\beta^{0}+\epsilon,$$
with $n=100$ response variables and $p=500$ unknown parameters. The design matrix $X$ is chosen with the rows being fixed i.i.d. realizations from $\mathcal{N}(0,\Sigma)$. 
Here the covariance matrix $\Sigma$ has a Toeplitz structure $$\Sigma_{i,j}=0.9^{|i-j|}.$$
We choose i.i.d. Gaussian errors $\epsilon$ with a variance of $\sigma^{2}=1$. 
For the underlying unknown parameter vector $\beta^{0}$ we choose different settings. For each such setting we calculate the square root LASSO and the square root SLOPE with the theoretical $\lambda$ given in this paper
and the $\lambda$ from a 8-fold Cross-validation on the mean squared prediction error.
We use $r=100$ repetitions to calculate the $\ell_1-$estimation error, the sorted $\ell_{1}-$estimation error and the $\ell_2-$prediction error.
As for the definition of the sorted $\ell_{1}-$norm, we chose a regular decreasing sequence from
$1$ to $0.1$ with length $500$.
The results can be found in Table \ref{m1},\ref{m2},\ref{m3} and \ref{m4}.

\underline{Decreasing Case:}\\
Here the active set is chosen as $S_{0}=\{1,2,3,...,7\}$, and \\$\beta^{0}_{S_{0}}=(4,\text{ }3.\overline{6},\text{ }3.\overline{3},\text{ }3,\text{ }2.\overline{6},\text{ }2.\overline{3},\text{ }2)^{T}$
is a decreasing sequence.
\begin{table}[H]
\centering
\caption{Decreasing $\beta$}
\label{m1}
\resizebox{\columnwidth}{!}{
\begin{tabular}{rcccccc}
 & \multicolumn{3}{c|}{theoretical $\lambda$}        & \multicolumn{3}{c}{Cross-validated $\lambda$} \\ 
 & $\lVert\beta^{0}-\hat{\beta}\rVert_{\ell_1}$& $J_{\lambda}(\beta^{0}-\hat{\beta})$ & \multicolumn{1}{l|}{$\lt{X(\beta^{0}-\hat{\beta})}$}&  
 $\lVert\beta^{0}-\hat{\beta}\rVert_{\ell_1}$ &  $J_{\lambda}(\beta^{0}-\hat{\beta})$     &   $\lt{X(\beta^{0}-\hat{\beta})}$ \\ \hline
 
 \text{\footnotesize srSLOPE}& 2.06     & 0.21     & \multicolumn{1}{c|}{\cc 4.12} & 2.37    &  0.26    & \cc 3.88 \\ \hline
 \text{\footnotesize srLASSO}& \cc 1.85 & \cc 0.19 & \multicolumn{1}{c|}{ 5.51}    &\cc 1.78 & \cc 0.19 & 5.05     \\ \hline
\end{tabular}
}
\end{table}

\underline{Decreasing Random Case:}\\
The active set was randomly chosen to be $S_{0}=\{154, 129, 276,  29, 233, 240, 402\}$ and again $\beta^{0}_{S_{0}}=(4,\text{ }3.\overline{6},\text{ }3.\overline{3},\text{ }3,\text{ }2.\overline{6},\text{ }2.\overline{3},\text{ }2)^{T}$.
\begin{table}[H]
\centering
\caption{Decreasing Random $\beta$}
\label{m2}
\resizebox{\columnwidth}{!}{
\begin{tabular}{rcccccc}
 & \multicolumn{3}{c|}{theoretical $\lambda$}        & \multicolumn{3}{c}{Cross-validated $\lambda$} \\ 
 & $\lVert\beta^{0}-\hat{\beta}\rVert_{\ell_1}$& $J_{\lambda}(\beta^{0}-\hat{\beta})$ & \multicolumn{1}{l|}{$\lt{X(\beta^{0}-\hat{\beta})}$}&  
 $\lVert\beta^{0}-\hat{\beta}\rVert_{\ell_1}$ &  $J_{\lambda}(\beta^{0}-\hat{\beta})$     &   $\lt{X(\beta^{0}-\hat{\beta})}$ \\ \hline
 
 \text{\footnotesize srSLOPE}& \cc 4.50 & \cc 0.49& \multicolumn{1}{c|}{\cc 7.74} &7.87     & 1.09     & \cc 7.68 \\ \hline
 \text{\footnotesize srLASSO}& 8.48     & 0.89    & \multicolumn{1}{c|}{29.47}    &\cc 7.81 & \cc 0.85 & 9.19     \\ \hline
\end{tabular}
}
\end{table}

\underline{Grouped Case:}\\
Now in order to see if the square root SLOPE can catch grouped variables better than the square root LASSO we look at
an active set $S_{0}=\{1,2,3,...,7\}$ together with $\beta^{0}_{S_{0}}=(4,4,4,3,3,2,2)^{T}$.
\begin{table}[H]
\centering
\caption{Grouped $\beta$}
\label{m3}
\resizebox{\columnwidth}{!}{
\begin{tabular}{rcccccc}
 & \multicolumn{3}{c|}{theoretical $\lambda$}        & \multicolumn{3}{c}{Cross-validated $\lambda$} \\ 
 & $\lVert\beta^{0}-\hat{\beta}\rVert_{\ell_1}$& $J_{\lambda}(\beta^{0}-\hat{\beta})$ & \multicolumn{1}{l|}{$\lt{X(\beta^{0}-\hat{\beta})}$}&  
 $\lVert\beta^{0}-\hat{\beta}\rVert_{\ell_1}$ &  $J_{\lambda}(\beta^{0}-\hat{\beta})$     &   $\lt{X(\beta^{0}-\hat{\beta})}$ \\ \hline
 
 \text{\footnotesize srSLOPE}& \cc 2.81 & \cc 0.29& \multicolumn{1}{c|}{\cc 6.43} &\cc 1.71& \cc 0.18 & \cc 3.65 \\ \hline
 \text{\footnotesize srLASSO}& 3.02     & 0.31    & \multicolumn{1}{c|}{8.37}     &1.83    & 0.19     & 4.25     \\ \hline
\end{tabular}
}
\end{table}

\underline{Grouped Random Case:}\\
Again we take the same randomly chosen set $S_{0}=\{154, 129, 276,  29, 233, 240, 402\}$ with
$\beta^{0}_{S_{0}}=(4,4,4,3,3,2,2)^{T}$.
\begin{table}[H]
\centering
\caption{Grouped Random $\beta$}
\label{m4}
\resizebox{\columnwidth}{!}{
\begin{tabular}{rcccccc}
 & \multicolumn{3}{c|}{theoretical $\lambda$}        & \multicolumn{3}{c}{Cross-validated $\lambda$} \\ 
 & $\lVert\beta^{0}-\hat{\beta}\rVert_{\ell_1}$& $J_{\lambda}(\beta^{0}-\hat{\beta})$ & \multicolumn{1}{l|}{$\lt{X(\beta^{0}-\hat{\beta})}$}&  
 $\lVert\beta^{0}-\hat{\beta}\rVert_{\ell_1}$ &  $J_{\lambda}(\beta^{0}-\hat{\beta})$     &   $\lt{X(\beta^{0}-\hat{\beta})}$ \\ \hline
 
 \text{\footnotesize srSLOPE}& \cc 6.05  & \cc 0.66& \multicolumn{1}{c|}{\cc 12.84} &\cc 5.80& \cc 0.66 & \cc 5.78 \\ \hline
 \text{\footnotesize srLASSO}& 16.90     & 1.77    & \multicolumn{1}{c|}{66.68}     &6.14    & 0.67     & 6.67     \\ \hline
\end{tabular}
}
\end{table}

The random cases usually lead to larger errors for both estimators. This is due to the correlation structure of the design matrix. 
The square root SLOPE seems to outperform the square root LASSO in the cases where
$\beta^{0}$ is somewhat grouped (grouped in the sense that amplitudes of same magnitude
appear). This is due to the structure of the sorted $\ell_{1}-$norm, which has some of the sparsity properties of $\ell_{1}$ as well as some of the grouping properties of $\ell_{\infty}$, 
see \citet{zeng}. Therefore the square root SLOPE reflects the underlying sparsity structure in the grouped cases.
What is also remarkable is that the square root SLOPE always has a better mean squared prediction error than the square root LASSO. This is even in cases, where square root LASSO has better
estimation errors. The estimation errors seem to be better for the square root LASSO in the decreasing cases.

\section{Discussion}
Sparsity inducing norms different from $\ell_{1}$ may be used to facilitate the interpretation of the results.
Depending on the sparsity structure we have provided sharp oracle inequalities for square root regularization. 
Due to the square root regularizing we do not need to estimate the variance, the estimators are all pivotal.
Moreover, because the penalty is a norm the optimization problems are all convex, which is a practical advantage when implementing the estimation procedures.
For these sharp oracle inequalities we only needed the weak decomposability and not the decomposability property of the $\ell_1-$norm. The weak decomposability generalizes the desired property
of promoting an estimated parameter vector with a sparse structure.
The structure of the $\Omega-$ and $\Omega^{S^{c}}-$norms influence the oracle bound. Therefore it is useful to use norms that reflect the true underlying sparsity structure.\\

\newpage
                             %
\bibliography{reference2}{}  %
\bibliographystyle{plainnat} %
                             %

\end{document}